\documentclass[12pt]{article}
\usepackage[a4paper, margin=1in]{geometry}
\usepackage[utf8]{inputenc}
\usepackage{fullpage}
\usepackage{todonotes}
\usepackage{url}

\usepackage{float}% figure float effect

\usepackage{caption} % caption for figures, auto number fig

\usepackage{graphicx} %添加图片

\usepackage{amsmath, amsthm}

\usepackage{amssymb}
\usepackage{hyperref}

\newtheorem{theorem}{Theorem}[section]
\newtheorem{corollary}[theorem]{Corollary}
\newtheorem{lemma}[theorem]{Lemma}
\newtheorem{definition}[theorem]{Definition}

\newtheorem{fact}[theorem]{Fact}

\newtheorem{claim}[theorem]{Claim}
\newtheorem{proposition}[theorem]{Proposition}

\begin{document}

	\title{Transversal tilings in $k$-partite graphs without large holes}
	
	\author{Xinyu He\thanks{ School of Mathematics, Shandong University, Jinan, China, Email: {\tt xyHe@mail.sdu.edu.cn}.}, 
Xiangxiang Nie\thanks{Data Science Institute, Shandong University, Jinan, China, Email: {\tt xiangxiangnie@sdu.edu.cn}.}, 
Donglei Yang\thanks{School of Mathematics, Shandong University, Jinan, China, Email: {\tt dlyang@sdu.edu.cn}.}
}
	\date{}
	\maketitle
	\begin{abstract}

We show that for any constant $\mu>0$ and $k\ge 3$, there exists $\alpha>0$ such that the following holds for sufficiently large $n \in \mathbb{N}$. If $G=(V_{1},\ldots,V_{k},E)$ is a spanning subgraph of the $n$-blow-up of $K_{k}$ with ${\delta^*}(G)\geq (\frac{1}{2}+\mu) n$ and $\alpha^*_{k-1}(G)<\alpha n$, then $G$ has a transversal $K_{k}$-factor. Moreover, the bound $\frac{1}{2}$ is asymptotically tight for the case \(k=3\). In addition, we show that if $k\ge 4$, $G=(V_{1},\ldots,V_{k},E)$ is a spanning subgraph of the $n$-blow-up of $C_{k}$ with ${\delta^*}(G)\ge (\frac{2}{k}+\mu) n$, and $\alpha^*_{2}(G)<\alpha n$, then $G$ has a transversal $C_{k}$-factor. This extends a recent result of Han, Hu, Ping, Wang, Wang and Yang.
\end{abstract}

	\section{Introduction}	

For a graph $F$ on $[k]:=\left\{ 1,\ldots,k\right\}$, we say that $B$ is the \textit{$n$-blow-up} of $F$ if there exists a partition $V_{1},\ldots,V_{k}$ of $V(B)$ such that $\vert V_{1}\vert=\cdots=\vert V_{k} \vert =n$ and $\{u,v\}\in E(B)$ if and only if $u\in V_{i}$ and $v\in V_{j}$ for some $\{i,j\}\in E(F)$.
Given a spanning subgraph $G$ of $B$, we call the sequence $V_1,\dots, V_k$ the \textit{parts} of $G$ and we define 
\begin{equation*}
{\delta^*}(G):=\min_{\{i,j\}\in E(F)}\delta(G[V_i,V_{j}]),
\end{equation*}
where $G[V_i,V_{j}]$ is the bipartite subgraph of $G$ induced by the parts $V_i$ and $V_{j}$. 
For a graph $H$, we call $\mathcal{H}$ an \textit{$H$-tiling} of $G$
if $\mathcal{H}$ consists of vertex disjoint copies of $H$ in $G$ and we say that $\mathcal{H}$ \textit{covers} 
$V(\mathcal{H}) := \bigcup\{ V(H') : H' \in \mathcal{H}\}$. We say that $\mathcal{H}$ is \textit{perfect} or an
\textit{$H$-factor} if it covers every vertex of $G$. A subset $R$ of $V(G)$ is \textit{balanced} if $\vert R\cap V_1\vert=\cdots=\vert R\cap V_k\vert$. In particular,
we say a subset of $V(G)$ or a subgraph of $G$ is \textit{transversal} if it intersects each part in exactly one vertex. 
An $\mathcal{H}$-tiling is \textit{transversal} 
if every copy of $H$ in $\mathcal{H}$ is transversal.

Fischer \cite{1999Variants} conjectured a multipartite version of the Hajnal-Szemer{\'e}di Theorem, stating that if $G$ is a spanning subgraph of the $n$-blow-up of $K_k$ and ${\delta^*}(G) \ge \left(1 - \frac{1}{k}\right)n$, then $G$ has a $K_k$-factor.
In the same paper, he proved that for $k\in \{3, 4\}$, such a graph $G$ contains a $K_k$-tiling of size at least $n - o(1)$.
 Johansson \cite{johansson2000triangle} established that for every $n$, if $G$ is a spanning subgraph of the $n$-blow-up of $K_3$ with ${\delta^*}(G) \ge 2n/3 + \sqrt{n}$, then $G$ contains a $K_3$-factor, thus proving the triangle case of the conjecture asymptotically. Later, Lo and M\"arkstrom \cite{LoMarkstrom} and, independently, Keevash and Mycroft \cite{KeevashMycroft2015} proved the conjecture asymptotically for all $k \ge 4$. The following theorem, which was established for $k=3$ by Magyar and Martin \cite{Magyar2002TripartiteVO}, for $k=4$ by Martin and Szemer\'edi \cite{martin2008}, and for $k \ge 5$ by Keevash and Mycroft \cite{KeevashMycroft2015}, demonstrates that Fischer's original conjecture was nearly true for $n$ sufficiently large. (Keevash and Mycroft in fact proved a stronger result, see Theorem 1.1 in \cite{KeevashMycroft2015} for details.)

\begin{theorem}\label{thm:fischer}
  For every $k$, there exists $n_0 := n_0(k)$ such that whenever 
  $n \ge n_0$ the following holds
  for every spanning subgraph $G$ of the $n$-blow-up of $K_k$ with $
\delta^*(G) \ge \left(1 - \frac{1}{k}\right)n.$
  The graph $G$ does not contain a $K_k$-factor if and only if 
  both $n$ and $k$ are odd, $k$ divides $n$ and $G$ is isomorphic to a specific
  spanning subgraph $\Gamma_{n,k}$ of the $n$-blow-up of $K_k$ 
  with $\delta^*(\Gamma_{n,k}) = \left(1 - \frac{1}{k}\right)n$. 
\end{theorem}

A recent trend in Ramsey-Tur\'{a}n Theory concerning graph tiling problems, is to impose an additional constraint on the independence number so as to lower the minimum degree threshold required in the original problem. A typical example is a result of Balogh, Molla, and Sharifzadeh \cite{balogh2016triangle}, which states that if $\alpha_{2}(G)=o(n)$ and $\delta(G) \ge n/2 + o(n)$, then $G$ contains a triangle-factor. This significantly weakens the bound $\delta(G) \ge 2n/3$ from the Corr\'adi--Hajnal theorem. The sharp major term $n/2$ can be seen by considering the union of two disjoint clique of size $n/2 - 1$ and $n/2 + 1$. For general $K_k$-factors with $k\ge 4$, Knierum and Su~\cite{KnierimSu2021} proved an asymptotically tight minimum degree condition $\delta(G) \ge(1-\frac{2}{k})n+o(n)$ under a small independence number $\alpha(G)=o(n)$. There have been more recent developments on this problems~\cite{CHANG2023301,CHEN2024373,chen2025cliquefactorsgraphslowkellindependence,HanHuWangYang2023CliqueFactors,HMWY}.

Recall that there have been a great many of minimum-degree-type tiling results for multipartite graphs~\cite{ergemlidze2022transversal, 1999Variants,johansson2000triangle, KeevashMycroft,LoMarkstrom,Magyar2002TripartiteVO,martin2008}; however, these do not incorporate any pseudo-randomness conditions. A natural follow-up question arises: what can we say if we impose some pseudo-randomness constraints on the multipartite graphs?
To formulate this, we consider an analogous independence-number condition that precludes the existence of large partite holes as follows, which is formally introduced by Nenadov and Pehova~\cite{nenadov2020ramsey}.

\begin{definition}\label{def:holes}
\rm{(\cite{nenadov2020ramsey}) For an integer $r \geq 2$, an \textit{$r$-partite hole} of size $s$ in a graph $G$ is a collection of $r$ disjoint vertex subsets $U_1,\ldots,U_r\subset V(G)$ of size $s$ such that there is no copy of $K_r$ in $G$ with one vertex in each of the $U_i$, $i\in[r]$.
For a graph $G$, $\alpha^*_r(G)$ denotes the size of the largest $r$-partite hole in $G$. 
}
\end{definition}
In particular, when $G$ itself is a $k$-partite graph, we additionally require in Definition~\ref{def:holes} that each $U_i$ should lie inside a partite set. In this case, it seems plausible to impose a small partite hole condition so as to weaken the minimum degree condition for example in Theorem~\ref{thm:fischer}.
Our first main result studies the minimum degree threshold for the existence of a transversal $K_k$-factor in $k$-partite graphs with sublinear $(k-1)$-partite holes and $k$-partite holes, respectively. This can be regarded as a generalization of the result of Balogh--Molla--Sharifzadeh on triangle factors in the partite setting.

	\begin{theorem}\label{thm1}
	Given a positive integer $k\geq3$ and positive constant $\delta$, there exists $\alpha>0$ such that the following holds for sufficiently large $n \in \mathbb{N} $. 
	Let $G=(V_{1},\ldots,V_{k},E)$ be  a spanning subgraph of the n-blow-up of $K_{k}$ with ${\delta^*}(G)\geq \delta n$ and $\alpha^*_{r}(G)<\alpha n$. If either of the following conditions is satisfied:
    \begin{enumerate}
        \item $\delta>\frac{1}{2}$ and $r=k-1$;
        \item $\delta>0$ and $r=k$.
    \end{enumerate}
    
    Then $G$ has a transversal $K_{k}$-factor.
	\end{theorem}

It is worth remarking that the degree condition in Theorem~\ref{thm1} part (1) is asymptotically tight for the case $k=3$, while the case $k\ge 4$ is still unclear. This is essentially derived from a construction given in \cite{HMWY}.

\begin{lemma}\label{020920} \rm{(\cite{HMWY})}
 For any $0<\alpha<1$ and all sufficiently large integers $n\in3\mathbb{N}$, letting $d=2\lceil(4/\alpha)^4\rceil$, there exists an $n$-vertex graph $G_0$ with $\delta(G_0)\ge n/2-2d^2$ and $\alpha^*_{\mathrm{2}}(G_0)<\alpha n$ such that $G_0$ contains no $K_3$-factor. 
\end{lemma}

\begin{corollary}\footnote{Such a graph could be obtained by uniformly and randomly splitting the graph $G_0$ as in Lemma~\ref{020920} into three equal parts and then using standard concentration inequalities.}
    For all $\alpha, \varepsilon>0$ such that the following holds for sufficiently large $n \in \mathbb{N} $. 
There exists a spanning subgraph $H$ of the $n$-blow-up of $K_{3}$ with ${\delta^*}(H)\geq (1/2-\varepsilon) n$, $\alpha^*_{\rm 2}(H)<\alpha n$ and $H$ contains no transversal $K_{3}$-factor.
\end{corollary}

In this direction for large cliques $K_k$ with $k\ge 4$, it would be interesting to determine the minimum degree threshold for transversal $K_k$-factors under the condition $\alpha^*_r(G)=o(n)$, where $r\le k$. In particular, for the case $r=2$, it seems plausible to show that $\delta^*(G)=(1-\tfrac{2}{k})n+o(n)$ is enough, as indicated by an aforementioned result of Knierum and Su~\cite{KnierimSu2021}, which is yet out of reach in this paper. 

Along this way, Han, Hu, Ping, Wang, Wang and Yang \cite{Han_Hu_Ping_Wang_Wang_Yang_2024} studied (among others) the transversal cycle factors and showed that if $G$ is a spanning subgraph of the $n$-blow-up of $C_k$ with $k \in 4\mathbb{N}$, ${\delta^*}(G) = \frac{2}{k}n+o(n)$ and $\alpha_{2}^*(G)=o(n)$, then $G$ contains a transversal $C_k$-factor. 
Our second result strengthens this as follows, where we do not require $k$ to be a multiple of $4$.

\begin{theorem}
\label{thm2}
        Given a positive integer $k\geq4$ and $\delta>\frac{2}{k}$, there exists $\alpha>0$ such that the following holds for sufficiently large $n \in \mathbb{N} $. 
	Let $G=(V_{1},\ldots,V_{k},E)$ be  a spanning subgraph of the n-blow-up of $C_{k}$ with ${\delta^*}(G)\geq \delta n$ and $\alpha^*_{2}(G)<\alpha n$. Then $G$ has a transversal $C_{k}$-factor.
\end{theorem}

Notice that if the condition $\alpha^*_2(G) = o(n)$ is removed from Theorem~\ref{thm2}, the minimum degree threshold for the existence of a transversal $C_k$-factor becomes asymptotically ${\delta^*}(G) \geq \left(1 + \frac{1}{k}\right)\frac{n}{2} + o(n)$, as established by Ergemlidze and Molla~\cite{ergemlidze2022transversal}. 
Moreover, the minimum degree bound in Theorem~\ref{thm2} may not be optimal. Actually, Han, Hu, Ping, Wang, Wang, and Yang \cite{Han_Hu_Ping_Wang_Wang_Yang_2024} conjectured that the (asymptotically) tight condition should be $\frac{n}{k}+o(n)$, as suggested by the so-called space barrier.
To see this, consider the following construction. Let $G$ be the complete $n$-blow-up of $C_k$ with parts $V_1, \ldots, V_k$. In each part $V_i$ for $i \in [k]$, select a subset $U_i \subseteq V_i$ of size $n/k - 1$, and let $U = \bigcup_{i \in [k]} U_i$. Remove from $G$ every edge that is not incident to $U$. Then, add a $k$-partite Erd\H{o}s graph (constructed via a random graph process) on the vertex set $V(G) \setminus U$, ensuring that the resulting graph $G'$ satisfies $\alpha^*_2(G') = o(n)$, while $G' - U$ contains no transversal copy of $C_k$.
Now, observe that every transversal $C_k$ in $G'$ must include at least one vertex from $U$, and the minimum partite degree satisfies ${\delta^*}(G') \geq n/k - 1$. However, since $|U| < n$, $G'$ has no transversal $C_k$-factor.

\paragraph{Notation.}
We begin by establishing notation and terminology that will be used throughout this paper. For a graph $G = (V, E)$, we denote $v(G) = |V(G)|$ and $e(G) = |E(G)|$. Given a subset $U \subseteq V$, let $G[U]$ be the subgraph of $G$ induced by $U$, and define $G - U := G[V \setminus U]$. For any two subsets $A, B \subseteq V(G)$, we write $E_G(A, B)$ for the set of edges between $A$ and $B$, and $e_G(A, B) := |E_G(A, B)|$. The neighborhood of a vertex $v$ in $G$ is denoted by $N_G(v)$, and we use $d_G(v)$ to denote the degree of $v$ in $G$ (i.e., the number of edges incident to $v$). We typically omit the subscript $G$ when the graph is clear from context. For integers $a \leq b$, we define $[a, b] := \{i \in \mathbb{Z} : a \leq i \leq b\}$.

When we write $\beta \ll \gamma$, we mean that $\beta$ and $\gamma$ are constants in $(0,1)$ and that there exists $\beta_0 = \beta_0(\gamma)$ such that all subsequent arguments hold for all $0 < \beta \leq \beta_0$. Hierarchies of other lengths are defined similarly.

\paragraph{Organization.}
The remainder of this paper is organized as follows. Section \ref{sec2} contains the proofs of our main results, Theorems \ref{thm1} and \ref{thm2}. In Section \ref{sec3}, we introduce the lattice-based absorbing method. Section \ref{sec4} is devoted to the regularity lemma and the construction of a tiling with specific structures in the reduced multigraph.

\section{Proofs of Theorem \ref{thm1} and \ref{thm2}}\label{sec2}

The proofs of Theorem \ref{thm1} and Theorem \ref{thm2} use the \emph{absorbing method}.
Our proof employs the absorption method and builds upon techniques developed in \cite{Han2016, Han2017Decision, KeevashMycroft2015}. 
The absorption method was originally introduced by R\"{o}dl, Ruci\'{n}ski, and Szemer\'{e}di approximately a decade ago in their seminal work \cite{Rodl2009}. 
Since its introduction, this method has emerged as a fundamental tool for investigating the existence of spanning structures in various combinatorial settings, including graphs, digraphs, and hypergraphs.

Following the standard absorption approach, the key objectives are to (i) define and construct an absorbing structure within the host graph capable of accommodating leftover vertices, and (ii) obtain an almost-perfect tiling that covers most vertices while leaving only a small linear proportion uncovered. Regarding objective (i), we note that the absorbing component of Theorem \ref{thm2} has been established by \cite{Han_Hu_Ping_Wang_Wang_Yang_2024}; consequently, this work focuses exclusively on proving the absorbing part of Theorem \ref{thm1}. To facilitate this, we now introduce several relevant definitions.

\begin{definition}{\rm(Absorber)}
	\rm {Let $F$ be a $k$-vertex graph and $G=(V_{1},\ldots,V_{k},E)$ be  a spanning subgraph of the $n$-blow-up of $F$.
Given a subset $S\subseteq V(G)$ of size $k$
		and an integer $t$, we say that a subset $A_{S}\subseteq V(G) \backslash S$ is an $(F,t)$-\textit{absorber} of $S$ if $\vert A_{S}\vert \leq kt$ and both $G[A_{S}]$ and $G[A_{S}\cup S] $ contain an $F$-factor.}
\end{definition}

\begin{definition}{\rm(Absorbing set)}
    \rm {Let $F$ be a $k$-vertex graph and
    $G=(V_{1},\ldots,V_{k},E)$
		be  a spanning subgraph of the $n$-blow-up of $F$.
	 We say that a balanced subset $R\subseteq V(G)$ is a $\xi$-\textit{absorbing}\textit{ set} for some $\xi>0$ if for every balanced  subset $U\subseteq V(G)\backslash R$ with $\vert U\vert\leq\xi n$,
		$G[R\cup U]$ contains an $F$-factor consisting of transversal $F$-copies.}
\end{definition}

Our first objective is to construct an absorbing set for Theorem~\ref{thm1}.

\begin{lemma}\label{3}
 Given $k \in \mathbb{N}$ with $k\geq3$ and positive constants $\delta,\gamma$ with $\gamma\leq\frac{\delta}{2}$,
	there exist $\alpha, \xi>0 $ such that the following holds for sufficiently large $n \in \mathbb{N} $. 
	Let $G=(V_{1},\ldots,V_{k},E)$ be  a spanning subgraph of the $n$-blow-up of $K_{k}$
	with ${\delta^*}(G)\geq \delta n$ and $\alpha^*_{r}(G)<\alpha n$.
If either of the following conditions is satisfied: 
        \begin{enumerate}
        \item $\delta>\frac{1}{2}$ and $r=k-1$;
        \item $\delta>0$ and $r=k$.
    \end{enumerate}
    
	Then there exists a $\xi$-absorbing set $R\subseteq V(G)$ of size at most $\gamma n$. 
\end{lemma}

For Theorem~\ref{thm2}, we construct an absorbing set using the result of Han, Hu, Ping, Wang, Wang, and Yang~\cite{Han_Hu_Ping_Wang_Wang_Yang_2024}, as follows.

\begin{lemma}\label{4}{\rm(\cite[Lemma~4.2]{Han_Hu_Ping_Wang_Wang_Yang_2024})}
	 Given $k \in \mathbb{N}$ with $k\geq4$ and positive constants $\delta,\gamma$ with $\delta>\frac{2}{k}$ and $\gamma\leq\frac{\delta}{2}$,
	there exist $\alpha, \xi>0 $ such that the following holds for sufficiently large $n \in \mathbb{N} $. 
	Let $G=(V_{1},\ldots,V_{k},E)$ be  a spanning subgraph of the $n$-blow-up of $C_{k}$
	with ${\delta^*}(G)\geq \delta n$ and $\alpha^*_{2}(G)<\alpha n$. 
	Then there exists a $\xi$-absorbing set $R\subseteq V(G)$ of size at most $\gamma n$. 
\end{lemma}

Lemmas~\ref{5} and~\ref{6} yield an almost perfect transversal $K_k$-tiling and an almost perfect transversal $C_k$-tiling, respectively.

\begin{lemma} \label{5}
	{\rm(Almost perfect tiling for Theorem \ref{thm1})}.
 Given $k \in \mathbb{N}$ with $k\geq3$ and positive constants $\delta,\zeta$ with $\delta>\frac{1}{2}$,
	there exists $\alpha>0 $ such that the following holds for sufficiently large $n \in \mathbb{N} $. 
	Let $G=(V_{1},\ldots,V_{k},E)$ be  a spanning subgraph of the $n$-blow-up of $K_{k}$
	with ${\delta^*}(G)\geq \delta n$ and $\alpha^*_{k-1}(G)<\alpha n$. 
    Then $G$ contains a transversal $K_{k}$-tiling covering all but at most $\zeta n$ vertices.
\end{lemma}

\begin{lemma} \label{6}
{\rm(Almost perfect tiling for Theorem \ref{thm2})}.
 Given $k \in \mathbb{N}$ with $k\geq4$ and positive constants $\delta,\zeta$ with $\delta>\frac{2}{k}$,
	there exists $\alpha>0 $ such that the following holds for sufficiently large $n \in \mathbb{N} $. 
	Let $G=(V_{1},\ldots,V_{k},E)$ be  a spanning subgraph of the $n$-blow-up of $C_{k}$
	with ${\delta^*}(G)\geq \delta n$ and $\alpha^*_{2}(G)<\alpha n$. 
	Then $G$ contains a transversal $C_{k}$-tiling covering all but at most $\zeta n$ vertices.

\end{lemma}

Now we are ready to prove  Theorem \ref{thm1} and Theorem \ref{thm2}.

\begin{proof}[Proof of Theorem \ref{thm1}]
	Given a constant $\delta$, we set $\eta:=\delta-\frac{1}{2}$ and
	choose
	$\frac{1}{n}\ll\alpha\ll\zeta\ll\xi\ll\gamma\ll\eta,\delta$.  
	Let $G=(V_{1},\ldots,V_{k},E)$ be  a spanning subgraph of the $n$-blow-up of $K_{k}$ with ${\delta^*}(G)\geq \delta n$ and $\alpha^*_{r}(G)<\alpha n$. 
	
	For case (1), we have $\delta>\frac{1}{2}$, $r=k-1$. By Lemma \ref{3} case (1) and 
	the choice that $\gamma \ll\eta,\delta$, there exists a $\xi$-absorbing set $R\subseteq V(G)$ of size at most $\gamma n$ for some $\xi>0$.
	Let $G^{\prime}:=G-R$ and note that $G^\prime$ is a spanning subgraph of $(n-\frac{\vert R\vert}{k})$-blow-up of $K_k$.
	Then we have 
    \begin{equation}
        {\delta^*}(G^{\prime})\geq \left(\frac{1}{2}+\eta\right) n-\frac{\gamma n}{k}\geq\left(\frac{1}{2}+\frac{\eta}{2} \right)
		n.\nonumber
    \end{equation}

	Therefore by applying Lemma \ref{5} on $G^{\prime}$, we obtain a transversal $K_{k}$-tiling $\mathcal{M}$ that covers all but a set $U$ of at most $\zeta n$ vertices
	in $G^{\prime}$. Since $\zeta\ll\xi$, the absorbing property of $R$ implies that
	$G[R\cup U]$ contains a transversal $K_{k}$-factor, which together with  $\mathcal{M}$  forms
	a transversal $K_{k}$-factor in $G$.

For case (2), we have $\delta>0$, $r=k$. By Lemma \ref{3} case (2) and 
	the choice that $\gamma \ll\eta,\delta$, there exists a $\xi$-absorbing set $R\subseteq V(G)$ of size at most $\gamma n$ for some $\xi>0$.
	Let $G^{\prime}:=G-R$ and note that $G^\prime$ is a spanning subgraph of $(n-\frac{\vert R\vert}{k})$-blow-up of $K_k$.
	Then we have 
    \begin{equation}
        {\delta^*}(G^{\prime})\geq \delta n-\frac{\gamma n}{k}\geq \frac{\delta}{2}
		n.\nonumber
    \end{equation}
    
Note that obtain an almost perfect tiling is almost trivial: by the definition of $\alpha^*_{k}(G)$, we can greedily pick as many vertex-disjoint transversal copies of $K_k$ as possible. The remaining subgraph is $K_k$-free and thus it has at most $k\alpha^*_{k}(G)\leq\zeta n$ vertices, denoted as $U$. Since $\zeta\ll\xi$, the absorbing property of $R$ implies that
	$G[R\cup U]$ contains a transversal $K_{k}$-factor, which together with  $\mathcal{M}$  forms
	a transversal $K_{k}$-factor in $G$.  
\end{proof}

\begin{proof}[Proof of Theorem \ref{thm2}]
	Given $k \in \mathbb{N}$ and a constant $\delta > \frac{2}{k}$, we define $\eta := \delta - \frac{2}{k}$ and choose parameters such that $\frac{1}{n} \ll \alpha \ll \zeta \ll \xi \ll \gamma \ll \eta, \delta.$
Let $G = (V_1, \ldots, V_k, E)$ be a spanning subgraph of the $n$-blow-up of $C_k$ satisfying ${\delta^*}(G) \geq \left( \frac{2}{k} + \eta \right) n$ and $\alpha^*_2(G) < \alpha n$.

By Lemma~\ref{4} and since $\gamma \ll \eta, \delta$, there exists a $\xi$-absorbing set $R \subseteq V(G)$ of size at most $\gamma n$ for some $\xi > 0$. 
Define $G' := G - R$, and note that $G'$ is a spanning subgraph of $\left(n - \frac{|R|}{k}\right)$-blow-up of $C_k$. 
We then have
\[
{\delta^*}(G') \geq \left( \frac{2}{k} + \eta \right) n - \frac{\gamma n}{k} \geq \left( \frac{2}{k} + \frac{\eta}{2} \right) n.
\]
Applying Lemma~\ref{6} to $G'$, we obtain a transversal $C_k$-tiling $\mathcal{B}$ that covers all but a set $U$ of at most $\zeta n$ vertices in $G'$. 
Since $\zeta \ll \xi$, the absorbing property of $R$ ensures that $G[R \cup U]$ contains a transversal $C_k$-factor. 
Combining this with $\mathcal{B}$ yields a transversal $C_k$-factor in $G$.
\end{proof}

\section{Absorbing sets}\label{sec3}

In this section, we prove Lemma~\ref{3} using the absorption method. A standard approach in absorption arguments for $F$-factors is to demonstrate that every $k$-set (where $k := |V(F)|$) admits polynomially many absorbers (see~\cite{Han_Hu_Ping_Wang_Wang_Yang_2024}). However, it remains uncertain whether this property holds in our context. We therefore employ an alternative construction that ensures the existence of a $\xi$-absorbing set under the weaker condition that \emph{every} $k$-set $S$ has \emph{linearly many vertex-disjoint $(F,t)$-absorbers}. This approach builds upon a key insight originally due to Montgomery~\cite{montgomery2019spanning}, which has subsequently been applied in numerous absorption-based proofs. Furthermore, a recent method developed by Nenadov and Pehova~\cite{nenadov2020ramsey} guarantees an absorbing set whenever every $k$-set possesses linearly many vertex-disjoint absorbers. Since the host graph in Lemma~\ref{3} is $k$-partite, our objective reduces to showing that every transversal $k$-set admits linearly many vertex-disjoint absorbers. To establish this, we will utilize the \emph{lattice-based absorbing method} introduced by Han~\cite{Han2017Decision}.

\subsection{obtaining absorbers}
To present the lattice-based absorbing method, we first introduce the necessary definitions and key lemmas that will be used in our analysis.

\begin{definition}
\rm{
Let $G$ and $F$ be as above, and let $m$, $t$ be positive integers. We say that two vertices $u, v \in V(G)$ are \textit{{$(F, m, t)$-reachable}} in $G$ if for every set $W \subseteq V(G)$ of size at most $m$, there exists a set $S \subseteq V(G) \setminus W$ with $|S| \leq kt - 1$ such that both $G[\{u\} \cup S]$ and $G[\{v\} \cup S]$ contain $F$-factors, where we call such $S$ an \emph{$F$-connector} for $u$, $v$. Moreover, a subset $U \subseteq V(G)$ is called \emph{$(F, m, t)$-closed} if every pair of vertices in $U$ is $(F, m, t)$-reachable in $G$ (note that the corresponding $F$-connectors are not required to be contained in $U$).
}
\end{definition}

The following result provides a sufficient condition under which every transversal $k$-set has a linear number of vertex-disjoint absorbers.
\begin{lemma}\label{7}
	Given a positive integer $k\geq3$ and a constant $\delta$,
	there exist $\alpha,\beta>0 $ such that the following holds for sufficiently large $n \in \mathbb{N} $.
	Let $G=(V_{1},\ldots,V_{k},E)$ be  a spanning subgraph of the n-blow-up of $K_{k}$ with ${\delta^*}(G)\geq \delta n$ and $\alpha^*_{r}(G)<\alpha n$. If either of the following conditions is satisfied:
    \begin{enumerate}
    \item  $\delta>\frac{1}{2}$ and $r=k-1$;
     \item  $\delta>0$ and $r=k$.
    \end{enumerate}
Then every transversal $k$-set in $G$ has at least $ \frac{\beta n-k}{4k^2}$
	vertex-disjoint $(K_{k},2k)$-absorbers. 
\end{lemma}

\subsection{Proof of Lemma \ref{3}} \label{13} 

To construct an absorbing set, we introduce the concept of an \emph{$F$-fan}.

\begin{definition}
	{\rm(\cite{HMWY})
		For a vertex $v\in V(G)$ and a $k$-vertex graph $F$, an \textit{$F$-fan} $\mathcal{F}_{v}$ at $v$
		in $V(G)$ is a collection of pairwise disjoint sets $S\subseteq  V(G)\backslash \left\{v \right\}$ such that for each $S\in \mathcal{F}_{v}$
		we have that $\vert S\vert=k-1$ and $\left\{v \right\}\cup S$ spans a copy of $F$.
		
	}
\end{definition}

We also employ the following bipartite template introduced by Montgomery \cite{montgomery2019spanning}.

\begin{lemma}\label{31}
{\rm(\cite{montgomery2019spanning})}
	Let $\beta>0$. There exists $m_{0}$ such that the following holds for every $m\geq m_{0}$.
	There exists a bipartite graph $B_{m}$ with vertex classes $X_{m}\cup Y_{m} $ and $Z_{m}$ and a maximum degree {\rm 40}, such that $\vert X_{m}\vert=m+\beta m$, $\vert Y_{m}\vert=2m$ and $\vert Z_{m}\vert=3m$, and for every subset $X^{\prime}_{m}\subseteq X_{m}$ of size $\vert X^{\prime}_{m}\vert=m$, the induced graph $B[X^{\prime}_{m}\cup Y_{m},Z_{m}]$ contains a perfect matching.
\end{lemma}

Now, we give the proof of Lemma \ref{3}, which follows from standard applications of  Lemma \ref{31} combined with Lemma \ref{7} as in for example \cite{Han_Hu_Ping_Wang_Wang_Yang_2024, HMWY}.
\begin{proof}[Proof of Lemma \ref{3}]

Given an integer $k \geq 3$ and constants $\delta$ and $\gamma \leq \frac{\delta}{2}$, we choose parameters such that
$
\frac{1}{n} \ll \alpha \ll \xi \ll \beta \ll \delta, \gamma, \frac{1}{k}.
$
Let $G = (V_1, \ldots, V_k, E)$ be a spanning subgraph of the $n$-blow-up of $K_k$ satisfying ${\delta^*}(G) \geq \delta n$ and $\alpha^*_{r}(G)<\alpha n$. By Lemma~\ref{7}, every transversal $k$-set in $G$ contains at least $\frac{\beta n - k}{4k^2}$ vertex-disjoint $(K_k, 2k)$-absorbers. Define $\tau = \frac{\beta}{8k^2}$; then for every $v \in V(G)$, there is a $K_k$-fan $\mathcal{F}_v$ in $V(G)$ of size at least $\frac{\beta n - k}{4k^2} \geq \tau n$. To complete the argument, it suffices to construct a $\xi$-absorbing set $R$ with $|R| \leq \tau n \leq \gamma n$ for some $\xi > 0$.
	
Let $q = \frac{\tau}{1000k^3}$ and $\beta' = \frac{q^{k-1}\tau}{2k}$. For each $i \in [k]$, let $X_i \subseteq V_i$ be a randomly chosen subset of size $qn$. For every vertex $v \in V(G)$, define $f_v$ to be the number of sets in $\mathcal{F}_v$ that are fully contained in $\bigcup_{i=1}^k X_i$. Note that the expected value satisfies $\mu:=\mathbb{E}[f_v] = q^{k-1}|\mathcal{F}_v| \geq q^{k-1}\tau n$. Applying the union bound and Chernoff's inequality, we obtain
\[
\mathbb{P}\left[{\rm there~is}~v \in V(G) \text{ with } f_v < \frac{\mu}{2}\right] \leq kn \exp\left(-\frac{(\mu/2)^2}{2\mu}\right) \leq kn \exp\left(-\frac{q^{k-1}\tau}{8}n\right) = o(1).
\]
Thus, for sufficiently large $n$, there exist subsets $X_i \subseteq V_i$ with $|X_i| = qn$ such that for every $v \in V(G)$, there is a subfamily $\mathcal{F}'_v \subseteq \mathcal{F}_v$ of at least $\frac{q^{k-1}\tau n}{2} = k\beta'n$ sets that are contained in $\bigcup_{i=1}^k X_i$.

	Let $m=\vert X_{i} \vert/(1+\beta^{\prime})$ and note that $m$ is linear in $n$. 
	Let $\left\{ I_{i}  \right\}_{i\in [k]}$ be a partition of $[3km]$ with each $\vert I_{i} \vert=3m$.
	For $i\in [k]$,	arbitrarily choose $k$ vertex-disjoint subsets $Y_{i}$,  $ Z_{i,j}$ for $j\in[k]\backslash \left\{ i\right\}$ in $V_{i}\backslash X_{i}$	with $\vert Y_{i}\vert=2m$ and $\vert Z_{i,j}\vert=3m$. Define $X = \bigcup_{i=1}^k X_i$, $Y = \bigcup_{i=1}^k Y_i$, and $Z = \bigcup_{\substack{i\in [k] \\ j\in [k]\setminus\{i\}}} Z_{i,j}$, noting that $|X| = (1+\beta')km$, $|Y| = 2km$, and $|Z| = 3k(k-1)m$. For each $j \in [k]$, partition $\bigcup_{i \in [k]\setminus\{j\}} Z_{i,j}$ into a family $\mathcal{Z}_j$ of $3m$ transversal $(k-1)$-sets and fix an arbitrary bijection $\phi_j: \mathcal{Z}_j \to I_j$. Define a function $\varphi: [3km] \to \bigcup_{j=1}^k \mathcal{Z}_j$ by setting $\varphi(x) = \phi_j^{-1}(x)$ for $x \in I_j$. For each $i \in [k]$, let $T_i$ be the bipartite graph from Lemma~\ref{31} with vertex classes $X_i \cup Y_i$ and $I_i$, and set $T = \bigcup_{i=1}^k T_i$. Then $T$ forms a bipartite graph between $X \cup Y$ and $[3km]$ with $\Delta(T) \leq 40$.
	
	\begin{claim}\label{1111}
There exists a family $\{A_e\}_{e \in E(T)}$ of pairwise vertex-disjoint subsets in $V(G) \setminus (X \cup Y \cup Z)$ such that for every edge $e = \{w_1, w_2\} \in E(T)$ with $w_1 \in X \cup Y$ and $w_2 \in [3km]$, the set $A_e$ forms a $(K_k, 2k)$-absorber for the transversal $k$-set $\{w_1\} \cup \varphi(w_2)$.
	\end{claim}
	\begin{proof}[Proof]

Suppose to the contrary that there exists an edge $e' \in E(T)$ for which no such subset $A_{e'}$ exists. Recall that $m = \frac{qn}{1+\beta'}$ and $\Delta(T) \leq 40$. Then we obtain the bound
\[
|X| + |Y| + |Z| + \left| \bigcup_{e \in E(T) \setminus \{e'\}} A_e \right| \leq 4km + 3k(k-1)m + k^2 |E(T)| \leq 4k^2 m + 80k^2 \cdot 3km \leq \frac{\tau n}{2}.
\]
Since every transversal $k$-set in $G$ has at least $\tau n$ vertex-disjoint $(K_k, 2k)$-absorbers, we may choose one such absorber in $V(G) \setminus \left( X \cup Y \cup Z \cup \bigcup_{e \in E(T) \setminus \{e'\}} A_e \right)$ to serve as $A_{e'}$, yielding a contradiction.
    	\end{proof}

Let $R=X \cup Y \cup Z\cup \bigcup_{e\in E(T)}A_{e}$. Then $\vert R \vert \leq \tau n$. To verify that $R$ is a $\xi$-absorbing set, consider an arbitrary balanced subset $U \subseteq V(G) \setminus R$ with $|U| \leq \xi n$. It suffices to show that $G[R \cup U]$ contains a transversal $K_k$-factor. Note that if there exist subsets $Q_i \subseteq X_i$ with $|Q_i| = \beta'm$ for each $i \in [k]$ such that $G\left[\bigcup_{i=1}^k Q_i \cup U\right]$ contains a transversal $K_k$-factor, then one can extend this to a full $K_k$-factor in $G[R \cup U]$. To see this, set $X_i' = X_i \setminus Q_i$. By Lemma~\ref{31}, there exists a perfect matching $M$ in $T$ between $\bigcup_{i=1}^k X_i' \cup Y$ and $[3km]$. For each edge $e = \{w_1, w_2\} \in M$, take a transversal $K_k$-factor in $G[\{w_1\} \cup \varphi(w_2) \cup A_e]$; for each $e' \in E(T) \setminus M$, take a transversal $K_k$-factor in $G[A_{e'}]$. These together form a transversal $K_k$-factor in $G\left[R \setminus \bigcup_{i=1}^k Q_i\right]$, which combined with the factor in $G\left[\bigcup_{i=1}^k Q_i \cup U\right]$ yields the desired transversal $K_k$-factor in $G[R \cup U]$.
	
Therefore, it remains to construct the sets $Q_i$ as described. Recall that for every vertex $v \in V(G)$, there exists a subfamily $\mathcal{F}_v'$ of at least $k\beta' n$ sets contained in $\bigcup_{i=1}^k X_i$. Since $\xi \ll \beta, \frac{1}{k}$, we have $|U| \leq \xi n \leq \beta' n$. Hence, we may greedily select a family $\mathcal{K}_1$ of vertex-disjoint transversal copies of $K_k$ that cover $U$, with all vertices lying in $\bigcup_{i=1}^k X_i$. Define $Q_{i1} = X_i \cap V(\mathcal{K}_1)$ for each $i \in [k]$; then $|Q_{i1}| = \frac{k-1}{k} |U|$, and $\mathcal{K}_1$ forms a transversal $K_k$-factor in $G\left[\bigcup_{i=1}^k Q_{i1} \cup U\right]$. Next, we greedily choose a family $\mathcal{K}_2$ of $\beta' m - \frac{k-1}{k} |U|$ vertex-disjoint transversal copies of $K_k$ in $G\left[X \setminus V(\mathcal{K}_1)\right]$. This is possible because every vertex $v$ has at least $k\beta' n - (k-1)|U| \geq k\left(\beta' m - \frac{k-1}{k} |U|\right)$ sets in $\mathcal{F}_v'$ that are disjoint from $V(\mathcal{K}_1)$. Now define $Q_{i2} = X_i \cap V(\mathcal{K}_2)$ and set $Q_i = Q_{i1} \cup Q_{i2}$ for each $i$. The sets $Q_i$ then satisfy the desired condition, since $\mathcal{K}_1 \cup \mathcal{K}_2$ is a transversal $K_k$-factor in $G\left[\bigcup_{i=1}^k Q_i \cup U\right]$. This completes the proof.
\end{proof}

\subsection{Proof of Lemma \ref{7}}

We will end this section with a detailed proof of Lemma \ref{7}.

\begin{proof}[Proof of Lemma \ref{7}]

Given a constant $\delta$, if $\delta > \frac{1}{2}$, we define $\eta := \delta - \frac{1}{2}$. Choose parameters satisfying $\frac{1}{n} \ll \alpha \ll \beta \ll \delta, \eta$. Let $G = (V_1, \ldots, V_k, E)$ be a spanning subgraph of the $n$-blow-up of $K_k$ such that ${\delta^*}(G) \geq \delta n$ and $\alpha^*_{r}(G)<\alpha n$.
	\begin{claim}\label{96}
		For $\delta>\frac{1}{2}$, $r=k-1$, $V_{i}$ is
		$(K_{k},\beta n,1)$-closed for each $i\in[k]$.
	\end{claim}
	\begin{proof}[Proof.]
		
Without loss of generality, assume that $i = 1$. For any two vertices $u, v \in V_1$, %since $\bar{\delta}(G) \geq \delta n$, we define $2(k-1)$ vertex subsets: for each $j \in [k] \setminus \{1\}$, 
let $D_{j,1} = N_{V_j}(u)$ and $D_{j,2} = N_{V_j}(v)$ for each $j \in [k] \setminus \{1\}$. Consider any vertex set $W \subseteq V(G) \setminus \{u, v\}$ with $|W| \leq \beta n$. For $i \in [k]$ and $j \in [k] \setminus \{1\}$, define $V'_i = V_i \setminus W$, $D'_{j,1} = D_{j,1} \setminus W$, and $D'_{j,2} = D_{j,2} \setminus W$. Note that $|D'_{j,1}|, |D'_{j,2}| \geq \delta n - \beta n \geq \left(\frac{1}{2} + \eta - \beta\right) n > \frac{n}{2}+\alpha n$, so we have $|D'_{j,1} \cap D'_{j,2}| \geq \alpha n$. For each $j \in [k] \setminus \{1\}$, let $S_j = D'_{j,1} \cap D'_{j,2}$. By the definition of $\alpha^*_{k-1}(G) < \alpha n$, there exists a transversal $K_{k-1}$. In fact, the transversal $K_{k-1}$ is the $K_k$-connector for $u,v$. It follows by definition that $u$ is $(K_k, \beta n, 1)$-reachable to $v$.	
	\end{proof}

	\begin{claim}\label{3456}
		For $\delta>0$, $r=k$, $V_{i}$ is
		$(K_{k},\beta n,2)$-closed for each $i\in[k]$.
	\end{claim}
	\begin{proof}[Proof.]
Without loss of generality, assume that $i = 1$. Consider any vertex set $W \subseteq V(G) \setminus \{u, v\}$ with $|W| \leq \beta n$. For $i \in [k]$ and $j \in [k] \setminus \{1\}$, define $V'_i = V_i \setminus W$, $D_{j,1} \subseteq N_{V_j}(u) \setminus W$, and $D_{j,2} \subseteq N_{V_j}(v) \setminus (W \cup D_{j,1})$ satisfying $|D_{j,l}| = 2\alpha n$ for $j \in [2,k]$, $l \in [2]$. In addition, 
let $A_u=\{w\in V_1' \mid \text{there is a transversal}~K_k \text{ in } \{w\}\cup(\bigcup_{j=2}^{k} D_{j1}) \}$ and $A_v=\{ w\in V_1'\mid \text{there is a transversal } K_k \text{ in } \{w\}\cup(\bigcup_{j=2}^{k} D_{j,2})\}.$
Note that $|A_{u}|, |A_{v}| \geq n-\alpha n-\beta n $, so we have $|A_{u} \cap A_{v}| \neq \emptyset$. Arbitrarily choose a vertex $w \in A_{u} \cap A_{v}$. Take vertices 
		$y_{j}\in D_{j1}$ for $j \in [k] \setminus \{1\}$ such that the set $K_1$ is a transversal $K_{k}$  passing through $w, y_{2},... ,y_{k}$. 
	Similarly, we can obtain another set $K_2$ of a transversal $K_{k}$ in $V(G)\backslash(W\cup V(K_{1}))$ that passes through $w, z_{2},..., z_{k}$, where $z_{j}\in D_{j,2}$ for $j \in [k] \setminus \{1\}$.
 	In fact, the set $\left\{ w \right\}\cup\left(V(K_{1})\backslash\left\{ u \right\}\right)\cup \left(V(K_{2})\backslash\left\{ v \right\}\right)$ is a $K_{k}$-connector for $u$, $v$. Therefore $u$ is $(K_{k},\beta n,2)$-reachable to $v$.
	\end{proof}

For every transversal $k$-subset $S\subseteq V(G)$, we greedily obtain as many pairwise disjoint $(K_{k},2k)$-absorbers for $S$ as possible.
For convenience,
we write $S=\left\{ s_{1}, s_{2},\ldots,s_{k} \right\}$ where $s_{i}\in V_{i}$ for $i\in[k]$.
Let $\mathcal{A}=\left\{ A_{1}, A_{2},\ldots,A_{\ell} \right\} $ be a maximal family of such absorbers. Suppose to the contrary that $\ell< \frac{\beta n-k}{4k^2}$. Since each $A_{j}$ has size at most $2k^2$, we have $\left\vert \bigcup_{j=1}^{\ell}A_{j}\right\vert< \frac{\beta n-k}{2}$.

Since $\alpha \ll \beta \ll \delta$, we can obtain a transversal copy $T$ of $K_k$ in $V(G) \setminus \left( \bigcup_{j=1}^{\ell} A_j \cup S \right)$, and write $T = \{t_1, t_2, \dots, t_k\}$ with $t_i \in V_i$ for each $i \in [k]$. By the closedness of each $V_i$, we select a family $\{I_1, I_2, \dots, I_k\}$ of vertex-disjoint subsets in $V(G) \setminus \left( \bigcup_{j=1}^{\ell} A_j \cup S \cup T \right)$ such that each $I_i$ is a $K_k$-connector for $s_i$ and $t_i$ with $|I_i| \leq 2k - 1$. Note that for any $1 \leq k' \leq k$, we have
\[
\left\lvert \bigcup_{j=1}^{\ell} A_j \cup \left( \bigcup_{i=1}^{k'} I_i \right) \cup S \cup T \right\rvert \leq \frac{\beta n - k}{2} + k(2k - 1) + 2k < \beta n.
\] 
Hence, we can iteratively choose each $I_i$ because $s_i$ and $t_i$ are $(K_k, \beta n, 2)$-reachable. It follows that $\bigcup_{i=1}^{k} I_i \cup T$ is a $(K_k, 2k)$-absorber for $S$, which contradicts the maximality of $\ell$.	
\end{proof}

\section{Almost cover}\label{sec4}
The proof of Lemma~\ref{4} follows from a standard application of the regularity method. 
%\subsection{Regularity}
We will first introduce some basic definitions and properties.
Given a graph $G$ and a pair $(X,Y)$ of vertex-disjoint subsets in $V(G)$, the
\textit{density} of $(X,Y)$ is defined as
\begin{equation}
	d(X,Y)=\frac{e(X,Y)}{\vert X \vert \vert Y \vert}. \nonumber
\end{equation}
Given constants $\varepsilon,d>0$, we say that $(X,Y)$ is \textit{$(\varepsilon,d)$-regular} if $d(X,Y)\geq d$ and for all $X^{\prime} \subseteq X$, $Y^{\prime} \subseteq Y$ with $\vert X^{\prime} \vert \geq \varepsilon \vert X \vert$ and $\vert Y^{\prime} \vert \geq \varepsilon \vert Y \vert$, we have
\begin{equation}
	\vert d(X^{\prime},Y^{\prime})-d(X,Y)\vert \leq \varepsilon. \nonumber
\end{equation}
The following fact is simply derived from the
definition.	
\begin{fact}\label{22}
Let $(X,Y)$ be an $(\varepsilon,d)$-regular pair and let $B \subseteq Y$ with $|B| \ge \varepsilon |Y|$. Then all but at most $\varepsilon |X|$ vertices in $X$ have at least $(d-\varepsilon)|B|$ neighbors in $B$.
\end{fact}
We now state a degree form of the regularity lemma (see
\cite{1991Szemer}, Theorem 1.10). 

\begin{lemma}\label{19}
	{\rm(Degree form of  Regularity Lemma \cite{1991Szemer}).} 
	For every $\varepsilon>0$ there is an $N = N(\varepsilon)$
	such that the following holds for any real number $d\in(0,1]$ and $n\in \mathbb{N}$. Let $G=(V, E)$ be an $n$-vertex graph.
	Then there exists a partition $\mathcal{P}=V_{0}\cup V_{1}\cup\dots\cup V_{k}$  
	and a spanning subgraph
	$G^{\prime} \subseteq G$ with the following properties: 
    \begin{enumerate}
	\item $\frac{1}{\varepsilon} \leq k \leq N$;
	  \item $\vert V_{0} \vert \leq \varepsilon n$ and
	$\vert V_{1} \vert=\dots=\vert V_{k} \vert=m\leq \varepsilon n$;
	\item $d_{G^{\prime}}(v)\geq d_{G}(v)-(d+\varepsilon)n$ for all $v \in V(G)$;
	\item every $V_{i}$ is an independent set in $G^{\prime}$ for $i\in[k]$;
	\item every pair $(V_{i}, V_{j})$, $1 \leq i<j \leq k$  is $\varepsilon$-regular in $G^{\prime}$ with density $0$ or at least $d$.
    \end{enumerate}
\end{lemma}
A fundamental tool used in conjunction with an $\varepsilon$-regular partition is the notion of reduced graph. The \emph{reduced graph} $R_d$ associated with the partition $\mathcal{P}$ is defined on the vertex set $\{V_1, \ldots, V_k\}$, where two parts $V_i$ and $V_j$ are connected by an edge in $R_d$ if the pair $(V_i, V_j)$ has density at least $d$ in $G'$. For each $i \in [k]$, we denote by $d_R(V_i)$ the degree of $V_i$ in $R_d$.

\begin{fact}\label{20}
Given positive constants $d$, $\varepsilon$, and $\delta$, consider an $n$-vertex graph $G = (V, E)$ with  $\delta(G) \geq \delta n$. Let $G'$ and $\mathcal{P}$ be obtained by applying Lemma{\rm{~\ref{19}}}, and let $R_d$ be the reduced graph as defined above. Then for every vertex $V_i \in V(R_d)$, we have $d_R(V_i) \geq (\delta - 2\varepsilon - d)k$.
\end{fact}

\subsection{Almost perfect tilings}\label{15}

We now proceed to the proofs of Lemma~\ref{5} and Lemma~\ref{6}, which will be established separately in the following.

\begin{proof}[Proof of Lemma \ref{5}]
Given constants $\delta, \zeta$ with $\delta > \frac{1}{2}$, define $\eta := \delta - \frac{1}{2}$ and choose parameters satisfying the hierarchy
$
\frac{1}{n} \ll \alpha \ll \frac{1}{N_0} \ll \varepsilon \ll \zeta, \delta, \eta.
$
Applying Lemma~\ref{19} to $G$ with $d := \frac{\eta}{4}$, we obtain a partition $\mathcal{P} = \{U_0\} \cup \{U_{i,j} \subseteq V_i : i \in [k], j \in [N_0]\}$ that refines the original partition $\{V_1, \dots, V_k\}$ of $G$, along with a spanning subgraph $G'$ satisfying conditions (a)--(e). Let $m := |U_{i,j}|$ for all $i \in [k]$ and $j \in [N_0]$. Define the reduced graph $R_d$ on the vertex set $\{U_{i,j} : i \in [k], j \in [N_0]\}$, where two vertices are adjacent if the corresponding pair has density at least $d$ in $G'$. For each $i \in [k]$, let $\mathcal{V}_i = \{U_{i,j} : j \in [N_0]\}$; then $\{\mathcal{V}_i : i \in [k]\}$ forms a partition of $V(R_d)$. By Fact~\ref{20}, we have ${\delta^*}(R_d) \geq \left(\delta - \frac{\eta}{2}\right)N_0 = \left(\frac{1}{2} + \frac{\eta}{2}\right)N_0$. 

To construct an almost perfect transversal $K_k$-tiling in $G$, we define an auxiliary transversal graph $H \cong K_{1,k-1}$ and utilize it to embed transversal copies of $K_k$ (see Claim~\ref{28}).
Given that ${\delta^*}(R_d) \geq \left(\frac{1}{2} + \frac{\eta}{2}\right) N_0$, there exists a perfect matching in the bipartite subgraph $R_d[\mathcal{V}_i, \mathcal{V}_{i+1}]$. This enables us to greedily select $N_0$ vertex-disjoint copies of $H$ that altogether cover all vertices of $R_d$. We denote by $\mathcal{H}$ this family of $N_0$ vertex-disjoint copies of $H$.

	\begin{claim}\label{28}
		For each copy of $H$ in $\mathcal{H}$, we can obtain a transversal $K_{k}$-tiling covering all but at most  $\frac{\zeta m}{2} $ vertices in the union of its clusters in $G$.
	\end{claim}
	
	\begin{proof}[Proof.]
    
Consider a copy of $H$ with vertex set say $V(H) = \{U_{1,1}, U_{2,2}, \ldots, U_{k,k}\}$, where $\{U_{1,1}, U_{j,j}\} \in E(H)$ for each $j \in [k] \setminus \{1\}$. This implies that the pair $(U_{1,1}, U_{j,j})$ is $(\varepsilon,d)$-regular in $G'$. To complete the argument, it suffices to show that for any collection of subsets $Z_j \subseteq U_{j,j}$ with $|Z_j| \geq \frac{\zeta m}{2k}$ for all $j \in [k]$, there exists a transversal copy of $K_k$ containing exactly one vertex from each $Z_j$.

Since $|Z_j| \geq \frac{\zeta m}{2k} \geq \varepsilon m$ for each $j \in [k]$, Fact~\ref{22} implies the existence of a subset $Z_1' \subseteq Z_1$ with $|Z_1'| \geq |Z_1| - \varepsilon m$ such that every vertex $x\in Z_1'$ has at least $(d - \varepsilon)|Z_j|$ neighbors in $Z_j$ for every $j \in [k] \setminus \{1\}$. Define $Z_j' = N(x) \cap Z_j$ for $j = 2,\ldots,k$. By the assumption that $a_{k-1}^*(G) < \alpha n$ and the lower bound $|Z_j'| \geq (d - \varepsilon)|Z_j| \geq \alpha n$ for each $j \in [k] \setminus \{1\}$, there exists at least one copy of $K_{k-1}$ among the sets $Z_2', \ldots, Z_k'$. Combining this with the vertex $x$ yields a transversal copy of $K_k$ in $\bigcup_{i=1}^k Z_i$. Repeating this procedure, we obtain a transversal $K_k$-tiling that covers all but at most $\frac{\zeta m}{2}$ vertices in the union of its clusters in $G$.
	\end{proof} 
	
This completes the proof, since the union of the $K_k$-tilings over all copies of $H$ in $\mathcal{H}$ leaves at most
\[
|U_0| + |\mathcal{H}| \cdot \frac{\zeta m}{2} \leq \varepsilon n + \frac{\zeta n}{2} \leq \zeta n
\]
vertices uncovered, as required.
\end{proof}	

To prove Lemma \ref{6}, we require the following proposition \ref{20021117}.

\begin{proposition}{\rm ({\cite{Han_Hu_Ping_Wang_Wang_Yang_2024}}, Proposition 4.7)}\label{20021117}
		Given an integer $k\ge2$ and a positive  constant $\alpha\leq\frac{1}{2}$,  let $G=(V_{1},\ldots,V_{k},E)$ be a spanning subgraph of the $n$-blow-up of $C_k$ with  $\alpha^*_{2}(G)<\alpha n$. 
		For any integers $i,j$ with $1\leq i<j\leq k$ and a collection of subsets $X_{s}\subseteq V_s$ with $s\in [i,j]$, if $\vert X_{i}\vert$, $\vert X_{j}\vert\geq\alpha n$ and $\vert X_{\ell}\vert\geq2\alpha n$ for $\ell\in [i+1,j-1]$ {\rm(}possibly empty{\rm)}, then there exists a transversal path   $x_{i}x_{i+1}\dots x_{j-1}x_{j}$   where $x_{s}\in X_{s}$ for $s\in [i,j]$.
	\end{proposition}

Next, we prove Lemma \ref{6}.

\begin{proof}[Proof of Lemma \ref{6}]

Given $k \in \mathbb{N}$ and constants $\delta, \zeta$ with $\delta > \frac{2}{k}$, define $\eta := \delta - \frac{2}{k}$ and choose parameters according to the hierarchy
$
\frac{1}{n} \ll \alpha \ll \frac{1}{N_0} \ll \varepsilon \ll \zeta, \delta, \eta.
$
Let $G = (V_1, \ldots, V_k, E)$ be a spanning subgraph of the $n$-blow-up of $C_k$ satisfying ${\delta^*}(G) \geq \delta n$ and $\alpha^*_2(G) < \alpha n$. Applying Lemma~\ref{19} to $G$ with $d := \frac{\eta}{4}$ yields a partition $\mathcal{P} = \{U_0\} \cup \{U_{i,j} \subseteq V_i : i \in [k], j \in [N_0]\}$ that refines the original partition $\{V_1, \dots, V_k\}$, along with a spanning subgraph $G'$ satisfying conditions (a)--(e). Let $m := |U_{i,j}|$ for all $i \in [k]$ and $j \in [N_0]$. Define the reduced graph $R_d$ on the vertex set $\{U_{i,j} : i \in [k], j \in [N_0]\}$, where two vertices are adjacent if the corresponding $\varepsilon$-regular pair has density at least $d$ in $G'$. For each $i \in [k]$, let $\mathcal{V}_i = \{U_{i,j} : j \in [N_0]\}$; then $\{\mathcal{V}_i : i \in [k]\}$ forms a partition of $V(R_d)$. By Fact~\ref{20}, we obtain ${\delta^*}(R_d) \geq \left(\delta - \frac{\eta}{2}\right)N_0 = \left(\frac{2}{k} + \frac{\eta}{2}\right)N_0$.
	
To construct an almost perfect transversal $C_k$-tiling in $G$, we define an auxiliary graph $H$ on $k$ vertices, which may take one of two forms: either $P^*_3$, consisting of a 3-vertex path together with $k-3$ isolated vertices, or $M_2$, consisting of two disjoint edges together with $k-4$ isolated vertices. Let $\mathcal{H}$ denote a family of vertex-disjoint copies of $H$.

Suppose that \(\mathcal{H}\) is a maximal tiling.
For each $i$, let $P_i= V(\mathcal{H})\cap V_i$, $L_i:=V_i\setminus P_i$ and \(L:=\bigcup_{i=1}^k L_i\). Note that  $\vert V_i \vert=N_0$ and $|P_i|=|P_j|$ for any $i,j\in [k]$.

	\begin{claim}\label{1117}
$L_i=\emptyset$ for any $i \in [k].$
	\end{claim}

\begin{proof}[Proof.]
Suppose for contradiction that $\vert L_i \vert>0$.
We first claim that $e(L)\le|L_1|^2$. Indeed, if $e(L)>|L_1|^2$, then there exist $L_i,L_j,L_p,L_q$ (maybe $j=p$) such that $e(L_i, L_j)>0$ and $e(L_p, L_q)>0$. 
Pick $u_s\in L_s$ with $s\in [k]$ such that $u_iu_j,u_pu_q\in E(G)$. Then we can obtain either a new $P^*_3$ or $M_2$ in $L$, contradicting to the maximality of \(\mathcal{H}\).  
Next we will count the number of edges between \(P\) and \(L\). Note that a lower bound is obtained as follows.
\[
\begin{aligned}
e(P,L)= \sum_{v\in L}\bigl(d(v)-d_L(v)\bigr)
&= \sum_{v\in L} d(v) - 2e(L) \\
&\ge 2\left(\frac{2}{k}N_0 + \frac{\eta}{2} N_0\right)|L| - 2\left(\frac{|L|}{k}\right)^2\\
&= |L|\left(\frac{4}{k}N_0 - \frac{2|L|}{k^2} + \eta N_0\right) \\
&= |L|\left(\frac{4}{k}\cdot \frac{|V(\mathcal{H})|+|L|}{k} - \frac{2|L|}{k^2} + \eta N_0\right) \\
&> |L|\left(\frac{4}{k^2}|V(\mathcal{H})| + \eta N_0\right).
\end{aligned}
\]

Without loss of generality, we assume that $\mathcal{H}$ consists of $p$ vertex-disjoint copies of $P^*_3$ and $q$ vertex-disjoint copies of $M_2$ for some $p,q\in\mathbb{N}$. Obviously, we have $p+q=|V(\mathcal{H})|/k$.

	\begin{fact}\label{0485}
For any copy in  \(\mathcal{H}\) of either $P_3^*$ or $M_2$, it holds that $e(P_3^*,L)\le 4|L|/k$ and $e(M_2,L)\le 4|L|/k$.
	\end{fact}

We defer its proof in Appendix \ref{claim}.
Taking it for granted, we have 
\[
\begin{aligned}
e(P,L)
&\le pe(P_3^*,L)+qe(M_2,L)\\
&\le\frac{4|P||L|}{k^2},
\end{aligned}
\]
yielding a contradiction with the lower bound. 
Therefore, $\vert L_i \vert  = 0$. 
\end{proof}

	\begin{claim}\label{1118}
		For each copy of $H\in \mathcal{H}$, we can obtain a transversal $C_{k}$-tiling covering all but at most  $\frac{\zeta m}{2} $ vertices in the union of its clusters in $G$.
	\end{claim}

\begin{proof}[Proof.]
For each copy of $H$ in $\mathcal{H}$, we consider separately the cases $H = M_2$ and $H = P^*_3$.

When $H = M_2$, we may assume without loss of generality that the two disjoint edges are separated by exactly one isolated vertex. Specifically, let $V(M_2) = \{U_{1,1}, U_{2,2}, \ldots, U_{k,k}\}$ with edges $\{U_{1,1}, U_{2,2}\}$ and $\{U_{4,4}, U_{5,5}\}$. Consequently, the pairs $(U_{1,1}, U_{2,2})$ and $(U_{4,4}, U_{5,5})$ are $(\varepsilon, d)$-regular in $G'$. To complete the argument, it suffices to show that for any subsets $Z_i \subseteq U_{i,i}$ with $|Z_i| \geq \frac{\zeta m}{2k}$ for each $i \in [k]$, there exists a transversal copy of $C_k$ containing exactly one vertex from each $Z_i$. Since $|Z_i| \geq \frac{\zeta m}{2k} \geq \varepsilon m$ for $i \in \{1,2,4,5\}$, Fact~\ref{22} implies the existence of a subset $Z_2' \subseteq Z_2$ with $|Z_2'| \geq |Z_2| - \varepsilon m$ such that every vertex in $Z_2'$ has at least $(d - \varepsilon)|Z_1|$ neighbors in $Z_1$, and similarly, a subset $Z_4' \subseteq Z_4$ with $|Z_4'| \geq |Z_4| - \varepsilon m$ such that every vertex in $Z_4'$ has at least $(d - \varepsilon)|Z_5|$ neighbors in $Z_5$. By the assumption $\alpha^*_{\mathrm{2}}(G) < \alpha n$ and noting that $|Z_2'|, |Z_4'| \geq \frac{\zeta m}{2k} - \varepsilon m \geq \alpha n$ and $|Z_i| \geq \frac{\zeta m}{2k} \geq 2\alpha n$ for every $i \in [k]$, Proposition~\ref{20021117} allows us to construct a copy of $P_3$ with one vertex in each of $Z_2'$, $Z_3$, and $Z_4'$.

Arbitrarily select a $P_3=\{u,w,v\}$ with $u \in Z_2'$, $w\in Z_3$ and $v \in Z_4'$. Define $Q_1 = N(u) \cap Z_1$ and $Q_2 = N(v) \cap Z_5$. Then $|Q_1|, |Q_2| \geq (d - \varepsilon) \cdot \frac{\zeta m}{2k} \geq \alpha n$, and note that $|Z_i| \geq \frac{\zeta m}{2k} \geq 2\alpha n$ for every $i \in [k]$. Again applying Proposition~\ref{20021117} with $X_i = Q_1$ and $X_j = Q_2$, we obtain a transversal path $T$ of length $(k-3)$ with endpoints $u' \in Q_1$ and $v' \in Q_2$. Then $T\cup T'$ is a transversal copy of $C_k$ in $\bigcup_{i=1}^k Z_i$, where $T'=u'uwvv'$. 
%By adding the edges $u'u$, $uw$, $wv$ and $vv'$, we can construct a transversal copy of $C_k$ in $\bigcup_{i=1}^k Z_i$.

When $H = P^*_3$, suppose we have a copy of $P^*_3$ in the reduced graph. Without loss of generality, assume its vertex set is $V(P^*_3) = \{U_{1,1}, U_{2,2}, \dots, U_{k,k}\}$ with edges $\{U_{1,1}, U_{2,2}\}$ and $\{U_{2,2}, U_{3,3}\}$. Then the pairs $(U_{1,1}, U_{2,2})$ and $(U_{2,2}, U_{3,3})$ are $(\varepsilon, d)$-regular in $G'$. To complete the argument, it suffices to show that for any subsets $Z_i \subseteq U_{i,i}$ with $|Z_i| \geq \frac{\zeta m}{2k}$ for each $i \in [k]$, there exists a transversal copy of $C_k$ containing exactly one vertex from each $Z_i$. Since $|Z_i| \geq \frac{\zeta m}{2k} \geq \varepsilon m$ for $i \in \{1,2,3\}$, Fact~\ref{22} implies that there exists a subset $Z_2' \subseteq Z_2$ with $|Z_2'| \geq |Z_2| - 2\varepsilon m$ such that every vertex in $Z_2'$ has at least $(d - \varepsilon)|Z_1|$ neighbors in $Z_1$ and at least $(d - \varepsilon)|Z_3|$ neighbors in $Z_3$.

  Choose a vertex $v \in Z_2'$, and define $Q_1 = N(v) \cap Z_1$ and $Q_2 = N(v) \cap Z_3$. Then $|Q_1|, |Q_2| \geq (d - \varepsilon) \cdot \frac{\zeta m}{2k} \geq \alpha n$, and we note that $|Z_i| \geq \frac{\zeta m}{2k} \geq 2\alpha n$ for all $i \in [k]$. Applying Proposition~\ref{20021117} with $X_i = Q_1$ and $X_j = Q_2$, we obtain a transversal path of length $k - 1$ with endpoints $u' \in Q_1$ and $v' \in Q_2$. By adding the edges $\{v, u'\}$ and $\{v, v'\}$, we obtain a transversal copy of $C_k$ in $\bigcup_{i=1}^k Z_i$. Repeating this procedure, we obtain a transversal $C_k$-tiling that covers all but at most $\frac{\zeta m}{2}$ vertices in the union of its clusters in $G$.
\end{proof}

	This would finish the proof as  the union of these $C_{k}$-tilings taken over all copies of $H$ in $\mathcal{H}$ would leave at most
	\begin{equation}
		\vert U_{0}\vert+\frac{\zeta m }{2}\vert {\mathcal{H}} \vert
		\leq \varepsilon n+\frac{\zeta n }{2}\leq\zeta n \nonumber
	\end{equation}
	vertices uncovered.
\end{proof}

\bibliographystyle{abbrv}
\bibliography{references}

@article{1991Szemer,
author = {Koml\'os, J. and Simonovits, M.},
year = {1996},
pages = {295-352},
title = {{S}zemer\'edi's Regularity Lemma and Its Applications in Graph Theory},
volume = {2},
journal = {Combinatorica}
}

@article{balogh2016triangle,
  title={Triangle factors of graphs without large independent sets and of weighted graphs},
  author={Balogh, J. and Molla, T. and Sharifzadeh, M.},
  journal={Random Struct. Algorithms},
  volume={49},
  number={4},
  pages={669--693},
  year={2016},
  publisher={Wiley Online Library}
}

@article{nenadov2020ramsey,
  title={On a {R}amsey--{T}ur{\'a}n Variant of the {H}ajnal--{S}zemer{\'e}di Theorem},
  author={Nenadov, R. and Pehova, Y.},
  journal={SIAM J. Discrete Math.},
  volume={34},
  number={2},
  pages={1001--1010},
  year={2020},
  publisher={SIAM}
}

@article{Han2017Decision,
  title={Decision problem for perfect matchings in dense $k$-uniform hypergraphs},
  author={Han, J.},
  journal={Trans. Am. Math. Soc.},
  volume={369},
  number={7},
  pages={5197-5218},
  year={2017},
}

@article{Magyar2002TripartiteVO,
  title={Tripartite version of the {C}orr{\'a}di--{H}ajnal theorem},
  author={ Magyar, C. and  Martin, R.},
  journal={Discrete Math.},
  year={2002},
  volume={254},
  pages={289-308}
}

@article{montgomery2019spanning,
  title={Spanning trees in random graphs},
  author={Montgomery, R.},
  journal={Adv. Math.},
  volume={356},
  pages={106793},
  year={2019},
  publisher={Elsevier}
}

@article{ergemlidze2022transversal,
  title={Transversal ${C}_k$-factors in subgraphs of the balanced blow-up of ${C}_k$},
  author={Ergemlidze, B. and Molla, T.},
  journal={Comb. Probab. Comput.},
  volume={31},
  number={6},
  pages={1031--1047},
  year={2022},
  publisher={Cambridge University Press}
}

@article{johansson2000triangle,
 title={Triangle-factors in a balanced blown-up triangle},
 author={Johansson, R.},
 journal={Discrete Math.},
 volume={211},
 number={1-3},
 pages={249--254},
 year={2000},
 publisher={Elsevier}
}

@article{1999Variants,
 title={Variants of the {H}ajnal--{S}zemer\'edi Theorem},
 author={ Fischer, E. },
 journal={J. Graph Theory},
 year={1999},
 volume={31},
 pages={275--282},
}

@article{KeevashMycroft,
	title={A multipartite {H}ajnal--{S}zemer{\'e}di theorem},
	author={Keevash, P. and Mycroft, R.},
	journal={J. Comb. Theory Ser. B},
	volume={114},
	pages={187--236},
	year={2015},
	publisher={Elsevier}
}

@article{LoMarkstrom,
	title={A multipartite version of the {H}ajnal--{S}zemer{\'e}di theorem for graphs and hypergraphs},
	author={Lo, A. and Markstr{\"o}m, K.},
	journal={Comb. Probab. Comput.},
	volume={22},
	number={1},
	pages={97--111},
	year={2013},
	publisher={Cambridge University Press}
}

@article{Han_Hu_Ping_Wang_Wang_Yang_2024,
title={Spanning trees in graphs without large bipartite holes}, 
volume={33},
DOI={10.1017/S0963548323000378},
number={3},
journal={Comb. Probab. Comput.}, 
author={Han, Jie and Hu, Jie and Ping, Lidan and Wang, Guanghui and Wang, Yi and Yang, Donglei},
year={2024}, 
pages={270–285}}

@article{martin2008,
  author  = {Martin, R. and Szemer{\'e}di, E.},
  title   = {Quadripartite version of the {H}ajnal--{S}zemer{\'e}di theorem},
  journal = {Discrete Math.},
  volume  = {308},
  pages   = {4337--4360},
  year    = {2008}
}

@article{HMWY,
  author  = {Han, J. and Morris, P. and Wang, G. and Yang, D.},
  title   = {A {R}amsey--{T}ur{\'a}n Theory for Tilings in Graphs},
  journal = {Random Struct. Algorithms},
  volume  = {64},
  year    = {2024},
  pages   = {94--124}
}

@article{chen2025cliquefactorsgraphslowkellindependence,
  author  = {Chen, Ming and Han, Jie and Yang, Donglei},
  title   = {Clique-factors in graphs with low {$K_{\ell}$}-independence number},
  year    = {2025},
  journal ={\rm{arXiv preprint arXiv:2509.16851}}

}

@article{Han2016,
  author  = {J. Han},
  title   = {Near perfect matchings in $k$-uniform hypergraphs {II}},
  journal = {SIAM J. Discrete Math.},
  year    = {2016},
  volume  = {30},
  pages   = {1453--1469}
}

@article{KeevashMycroft2015,
  author  = {P. Keevash and R. Mycroft},
  title   = {A geometric theory for hypergraph matching},
  journal = {Memoirs Am. Math. Soc.},
  year    = {2015},
  volume  = {223},
  number  = {1098},
  pages   = {vi+95}
}

@article{Rodl2009,
  author  = {V. R\"{o}dl and A. Ruci{\'n}ski and E. Szemer{\'e}di},
  title   = {Perfect matchings in large uniform hypergraphs with large minimum collective degree},
  journal = {J. Comb. Theory Ser. A},
  year    = {2009},
  volume  = {116},
  number  = {3},
  pages   = {613--636}
}

@article{HanHuWangYang2023CliqueFactors,
  author  = {Han, Jie and Hu, Ping and Wang, Guanghui and Yang, Donglei},
  title   = {Clique-factors in graphs with sublinear $\ell$-independence number},
  journal = {Comb. Probab. Comput.},
  year    = {2023},
  volume  = {32},
  number  = {4},
  pages   = {665--681},
}

@article{CHANG2023301,
title = {Embedding clique-factors in graphs with low $\ell$-independence number},
journal = {J. Comb. Theory Ser. B},
volume = {161},
pages = {301-330},
year = {2023},
author = {Fan Chang and Jie Han and Jaehoon Kim and Guanghui Wang and Donglei Yang},

}

@article{CHEN2024373,
title = {H-factors in graphs with small independence number},
journal = {J. Comb. Theory Ser. B},
volume = {169},
pages = {373-405},
year = {2024},
author = {Ming Chen and Jie Han and Guanghui Wang and Donglei Yang},
}

@article{KnierimSu2021,
  author  = {Knierim, Charlotte and Su, Pascal},
  title   = {${K}_r$-factors in graphs with low independence number},
  journal = {J. Comb. Theory Ser. B},
  volume  = {148},
  pages   = {60--83},
  year    = {2021},
}

\appendix
\section{Proof of Fact \ref{0485}}\label{claim}

In the proof of Fact \ref{0485}, we perform a case analysis on the number of edges between  $L$  and the non-isolated vertices inside each copy of either $P_3^*$ or $M_2$ in $\mathcal{H}$.
Let $P^* = u_2u_3u_4 \cup \{u_1, u_5, \ldots, u_k\}$ be a copy of $P^*_3$ in $\mathcal{H}$.  
Let $M = u_2u_3 \cup u_4u_5 \cup \{u_1, u_6, \ldots, u_k\}$ and  
$M' = u_2u_3 \cup u_5u_6 \cup \{u_1, u_4, u_7, \ldots, u_k\}$ be two copies of $M_2$ in $\mathcal{H}$.
Let $\overrightarrow{e}(u_i, L)$ ( or $\overleftarrow{e}(u_i, L)$) denote the number of edges between $u_i$ and $L_{i+1}$ (resp. $L_{i-1}$), where $L_0 = L_k$ and $L_{k+1} = L_1$.  
Define $e(u_i, L) = \overrightarrow{e}(u_i, L) + \overleftarrow{e}(u_i, L)$.
Let $I_{P^*}$, $I_M$, and $I_{M'}$ denote the sets of non-isolated vertices of $P^*$, $M$, and $M'$ in $\mathcal{H}$, respectively.  
Let $J_{P^*}$, $J_M$, and $J_{M'}$ denote the sets of isolated vertices of $P^*$, $M$, and $M'$ in $\mathcal{H}$, respectively.

\begin{claim}\label{P^*}
    Let $P^*$ be the set defined above. Then the following statements hold:
\begin{enumerate}
    \item[$(1)$.] If $\overleftarrow{e}(u_3,L) > 0$ and $\overrightarrow{e}(u_3,L) > 0$, then for each $i \in \{2,4\}$, at most one of $\overleftarrow{e}(u_i, L) > 0$ and $\overrightarrow{e}(u_i, L) > 0$ holds.
    \item[$(2)$.] Assume $e(I_{P^*}, L) > 3|L|/k$. If $\overleftarrow{e}(u_3, L) > 0$ and $\overrightarrow{e}(u_3, L) > 0$, then at most one of the conditions $\overrightarrow{e}(u_2, L) > 0$ and $\overleftarrow{e}(u_4, L) > 0$ holds. 
    Moreover, we have at most one of the following four conditions holds:
 $\overleftarrow{e}(u_2, L) > 0,  \overrightarrow{e}(u_2, L) > 0, \overleftarrow{e}(u_4, L) > 0, \overrightarrow{e}(u_4, L) > 0.$
    \label{1.1}    
    \item[$(3)$.] If $\overleftarrow{e}(u_2, L) > 0$ and $\overrightarrow{e}(u_2, L) > 0$, then $\overleftarrow{e}(u_3, L) = \overrightarrow{e}(u_4, L) = 0$.
    \label{1.2}
\end{enumerate}

\end{claim}

\begin{proof}
For $(1)$, assume that $\overleftarrow{e}(u_3, L) > 0$ and $\overrightarrow{e}(u_3, L) > 0$, and suppose that both $e(u_2, L) > 0$ and $e(u_4, L) > 0$ hold. Take $v_2 \in N(u_3) \cap L_2$ and $v_4 \in N(u_3) \cap L_4$.

By symmetry, assume that $\overleftarrow{e}(u_2, L) > 0$ and $\overrightarrow{e}(u_2, L) > 0$. Take $v_1 \in N(u_2) \cap L_1$ and $v_3 \in N(u_2) \cap L_3$. We can obtain a new $\{P_3^*, M_2\}$-tiling $\mathcal{H}^*$ by replacing $P^*$ with $v_2 u_3 v_4 \cup \{u_1, u_5, \dots, u_k\}$ and $v_1 u_2 v_3 \cup \{u_4, v_5, \dots, v_k\}$, where $v_s \in L_s$. Clearly, $|V(\mathcal{H}^*)| = |V(\mathcal{H})| + k$, which contradicts the maximality of $\mathcal{H}$. Hence, for each $i \in \{2,4\}$, at most one of $\overleftarrow{e}(u_i, L) > 0$ and $\overrightarrow{e}(u_i, L) > 0$ holds.

For $(2)$, if $\overrightarrow{e}(u_2, L) > 0$ and $\overleftarrow{e}(u_4, L) > 0$, then by $e(I_{P^*}, L) > 3|L|/k$, we have $L_3 \cap N(u_2) \cap N(u_4) \neq \emptyset$. Pick $v_3 \in L_3 \cap N(u_2) \cap N(u_4)$. We can obtain a new $\{P^*_3, M_2\}$-tiling $\mathcal{H}'$ by replacing $P^*$ with $v_2 u_3 v_4 \cup \{u_1, u_5, \dots, u_k\}$ and $u_2 v_3 u_4 \cup \{v_1, v_5, \dots, v_k\}$, where $v_s \in L_s$. Clearly, $|V(\mathcal{H}')| = |V(\mathcal{H})| + k$ contradicts the maximality of $\mathcal{H}$.
For the moreover part,  Suppose for contradiction we can pick $v_i \in N(u_2) \cap L$ and $v_j \in N(u_4) \cap L$. We can obtain a new $\{P^*_3, M_2\}$-tiling $\mathcal{H}''$ by replacing $P^*$ with $v_2 u_3 v_4 \cup \{u_1, u_5, \dots, u_k\}$ and $u_2 v_i \cup u_4 v_j \cup \bigl( \{v_1, \dots, v_k\} \setminus \{v_i, v_j, v_2, v_4\} \bigr)$, where $v_s \in L_s$. Clearly, $|V(\mathcal{H}'')| = |V(\mathcal{H})| + k$ contradicts the maximality of $\mathcal{H}$.

\begin{figure*}[htbp]
    \centering
    \includegraphics[trim=4cm 5.6cm 4cm 4.2cm,clip, scale=0.5]{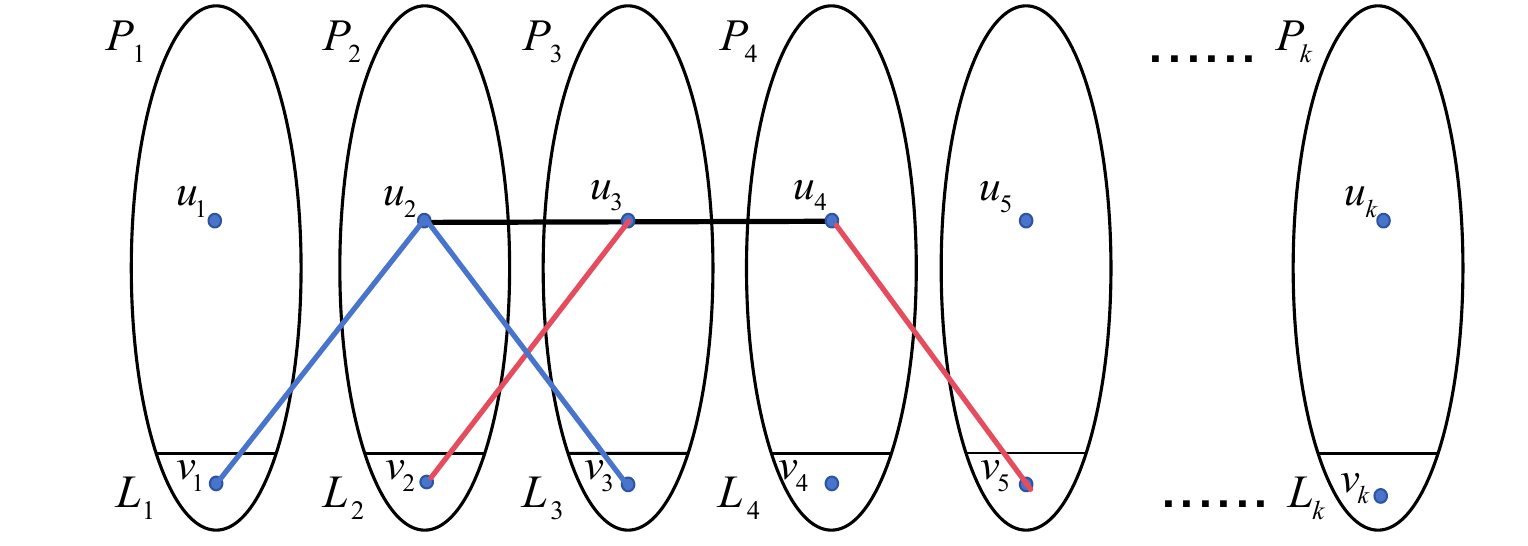}
    \caption{ $P^*=u_2u_3u_4\cup\{u_1,u_5,\cdots, u_k\}$}
    \label{fig2}
\end{figure*}

For $(3)$, suppose $\overleftarrow{e}(u_2, L) > 0$ and $\overrightarrow{e}(u_2, L) > 0$. Take $v_1 \in N(u_2) \cap L_1$ and $v_3 \in N(u_2) \cap L_3$. Now suppose that either $\overleftarrow{e}(u_3, L) > 0$ or $\overrightarrow{e}(u_4, L) > 0$. Pick $v_2 \in N(u_3) \cap L_2$ or $v_5 \in N(u_4) \cap L_5$; see Figure~\ref{fig2}. Then we obtain a new $\{P^*_3, M_2\}$-tiling $\mathcal{H}'$ by replacing $P^*$ either with $v_1 u_2 v_3 \cup \{v_4, u_5, \dots, u_k\}$ and $v_2 u_3 u_4 \cup \{u_1, v_5, \dots, v_k\}$, or with $v_1 u_2 v_3 \cup \{v_4, u_5, \dots, u_k\}$ and $u_3 u_4 v_5 \cup \{u_1, v_2, v_6, \dots, v_k\}$, where each $v_s \in L_s$. Clearly, $|V(\mathcal{H}')| = |V(\mathcal{H})| + k$, which contradicts the maximality of $\mathcal{H}$.
\end{proof}

\begin{claim}\label{M}
 Let $M$ be the set defined above. Then the following conclusions hold:
\begin{enumerate}%[label=p1]
    \item[$(1)$.] If $\overleftarrow{e}(u_2,L)>0$ or $\overrightarrow{e}(u_3,L)>0$, then $\overleftarrow{e}(u_4,L)=\overrightarrow{e}(u_5,L)=0$.\label{2.1}
    \item[$(2)$.] If $\overleftarrow{e}(u_3,L)>0$ and $\overrightarrow{e}(u_3,L)>0$, then $e(u_2,L)=0$.\label{2.2}
    \item[$(3)$.] If $\overleftarrow{e}(u_2,L)>0$ and $\overrightarrow{e}(u_2,L)>0$, then $\overleftarrow{e}(u_3,L)=0$.\label{2.3}
\end{enumerate}
\end{claim}

\begin{proof}
For $(1)$, assume that $\overrightarrow{e}(u_3, L) = |L|/k$. Choose $v_4 \in N(u_3) \cap L_4$. 
If $\overleftarrow{e}(u_4,L)>0$, choose $v_3 \in N(u_4) \cap L_3$; see Figure~\ref{fig3}.
Then we obtain a new $\{P^*_3, M_2\}$-tiling $\mathcal{H}'$ by replacing $M$
with $u_2 u_3 v_4 \cup \{u_1, v_5, u_6, \dots, u_k\}$ and $v_3 u_4 u_5 \cup \{v_1, v_2, v_6, \dots, v_k\}$.
If $\overrightarrow{e}(u_5,L)>0$, choose $v_6 \in N(u_5) \cap L_6$.
Then we obtain a new $\{P^*_3, M_2\}$-tiling $\mathcal{H}'$ by replacing $M$  with $u_2 u_3 v_4 \cup \{u_1, v_5, u_6, \dots, u_k\}$ and $u_4 u_5 v_6 \cup \{v_1, v_2, v_3, v_7, \dots, v_k\}$, where $v_s \in L_s$. Clearly, $|V(\mathcal{H}')| = |V(\mathcal{H})| + k$ contradicts the maximality of $\mathcal{H}$. The proof is similar for the case $\overleftarrow{e}(u_2, L) > 0$.

\begin{figure*}[htbp]
    \centering
    \includegraphics[trim=4cm 5.7cm 3.7cm 4cm,clip, scale=0.5]{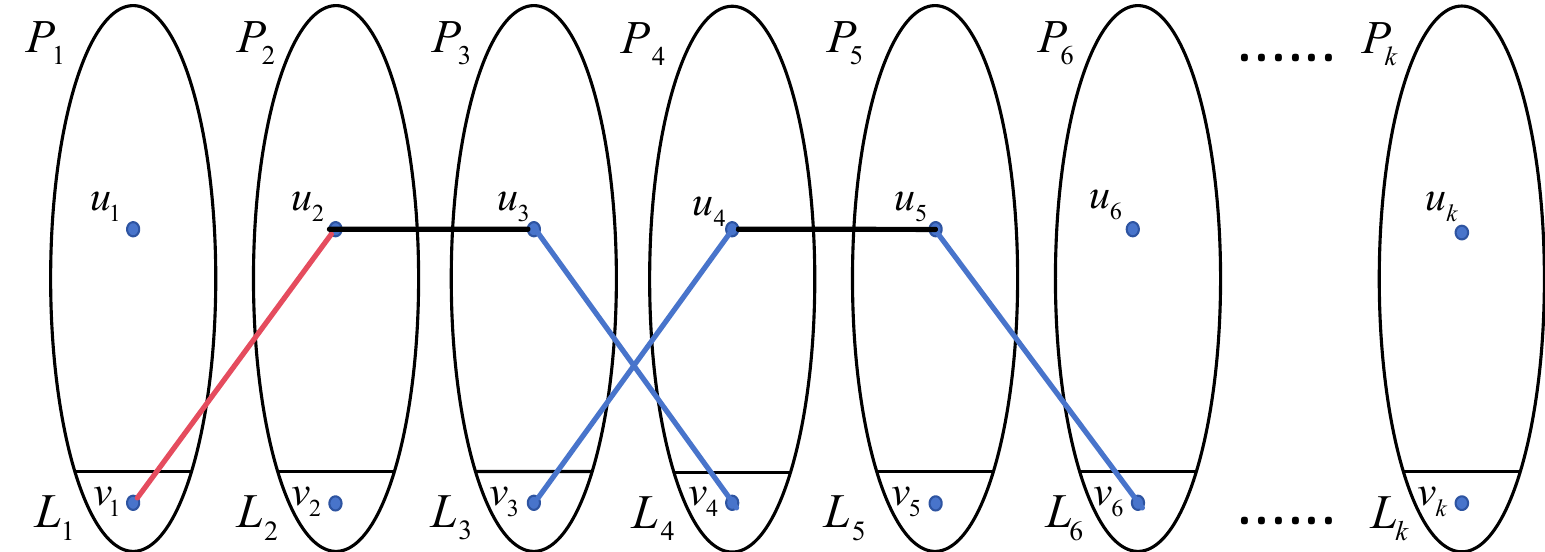}
    \caption{ $M=u_2u_3\cup u_4u_5\cup\{u_1,u_6,u_7\ldots,u_k\}$}
    \label{fig3}
\end{figure*}

\begin{figure*}[htbp]
    \centering
    \includegraphics[trim=4cm 5.7cm 3.7cm 4cm,clip, scale=0.5]{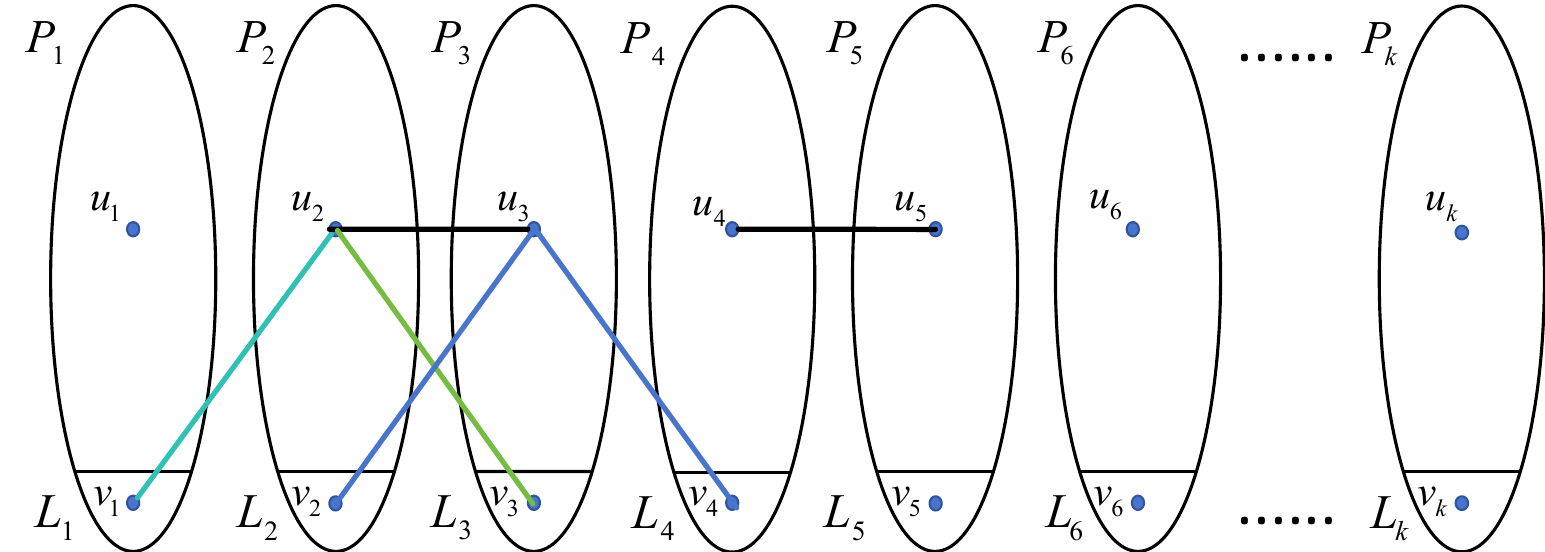}
    \caption{ $M=u_2u_3\cup u_4u_5\cup\{u_1,u_6,u_7\ldots,u_k\}$}
    \label{fig4}
\end{figure*}

For $(2)$, assume that $\overleftarrow{e}(u_3, L) > 0$ and $\overrightarrow{e}(u_3, L) > 0$. Take $v_2 \in N(u_3) \cap L_2$ and $v_4 \in N(u_3) \cap L_4$. Suppose that there exists a vertex $v_i \in N(u_2) \cap L_i$ for some $i \in \{1,3\}$; see Figure~\ref{fig4}. Then we obtain a new $\{P^*_3, M_2\}$-tiling $\mathcal{H}''$ by replacing $M$ with $v_2 u_3 v_4 \cup \{u_1, v_5, u_6, \dots, u_k\}$ and $v_i u_2 \cup u_4 u_5 \cup \bigl( \{v_1, v_3, v_6, \dots, v_k\} \setminus \{v_i\} \bigr)$, where $v_s \in L_s$. Clearly, $|V(\mathcal{H}'')| = |V(\mathcal{H})| + k$ contradicts the maximality of $\mathcal{H}$.

\begin{figure*}[htbp]
    \centering
    \includegraphics[trim=4cm 5.7cm 3.7cm 4cm,clip, scale=0.5]{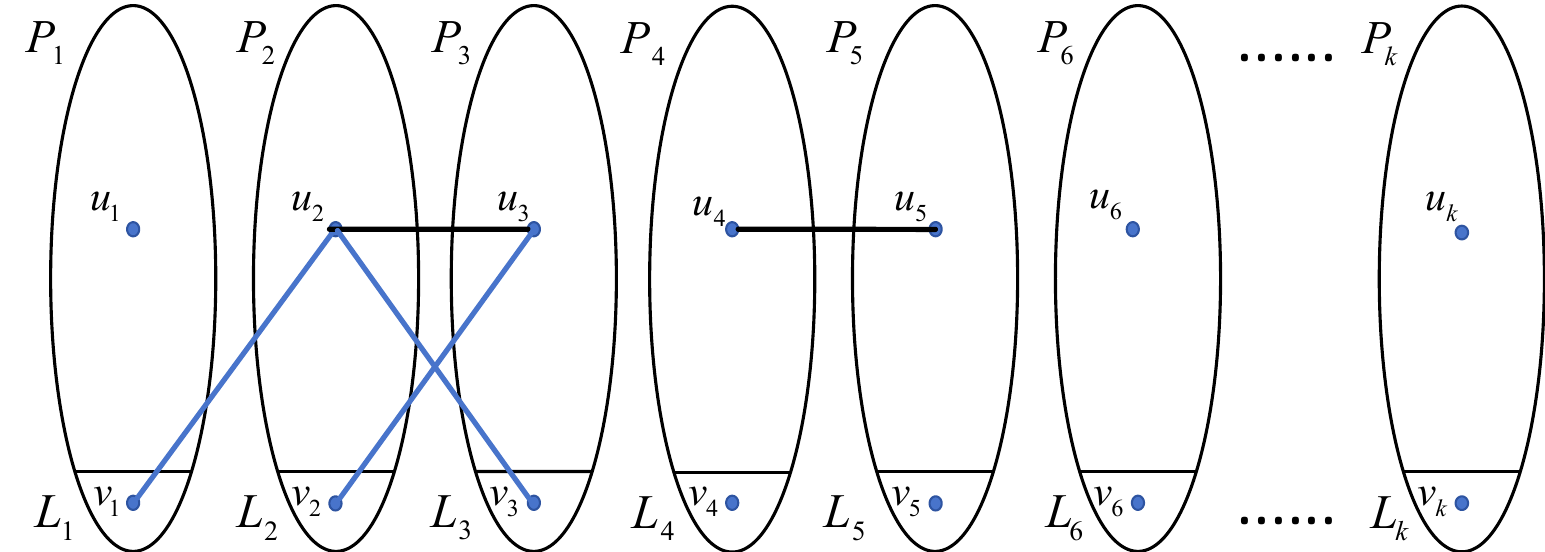}
    \caption{ $M=u_2u_3\cup u_4u_5\cup\{u_1,u_6,u_7\ldots,u_k\}$}
    \label{fig5}
\end{figure*}

For $(3)$, assume that $\overleftarrow{e}(u_2, L) > 0$ and $\overrightarrow{e}(u_2, L) > 0$. Take $v_1 \in N(u_2) \cap L_1$ and $v_3 \in N(u_2) \cap L_3$. Suppose that there exists a vertex $v_2 \in N(u_3) \cap L_2$; see Figure~\ref{fig5}. Then we obtain a new $\{P^*_3, M_2\}$-tiling $\mathcal{H}'''$ by replacing $M$ with $v_1 u_2 v_3 \cup \{v_4, v_5, u_6, \dots, u_k\}$ and $v_2 u_3 \cup u_4 u_5 \cup \{u_1, v_6, \dots, v_k\}$, where $v_s \in L_s$. Clearly, $|V(\mathcal{H}''')| = |V(\mathcal{H})| + k$ contradicts the maximality of $\mathcal{H}$.
\end{proof}

\begin{claim}\label{M'}
Let $M'$ be the set as above. Then the following conclusions hold:
\begin{enumerate}%[label=p1]
    \item[$(1)$.] If $\overleftarrow{e}(u_2,L)>0$, then $\overleftarrow{e}(u_5,L)=\overrightarrow{e}(u_6,L)=0$; If $\overrightarrow{e}(u_3,L)>0$, then $\overrightarrow{e}(u_6,L)=0$.\label{3.1}
    \item[$(2)$.] If $\overleftarrow e(u_3,L)>0$ and $\overrightarrow e(u_3,L)>0$, then $e(u_2,L)=0$.\label{3.2}
    \item[$(3)$.] If $\overleftarrow e(u_2,L)>0$ and $\overrightarrow e(u_2,L)>0$, then $e(u_3,L)=0$.\label{3.3}
\end{enumerate}
\end{claim}
\begin{proof}
The proof is similar to the proof of Claim \ref{M} and thus omitted except $\overrightarrow{e}(u_3,L)=0$ in case (3).

Now assume that $\overleftarrow{e}(u_2, L) > 0$ and $\overrightarrow{e}(u_2, L) > 0$. Take $v_1 \in N(u_2) \cap L_1$ and $v_3 \in N(u_2) \cap L_3$. Suppose that there exists a vertex $v_4 \in N(u_3) \cap L_4$; see Figure~\ref{fig0126}.
\begin{figure*}[htbp]
    \centering
    \includegraphics[trim=4cm 5.7cm 3.7cm 4cm,clip, scale=0.5]{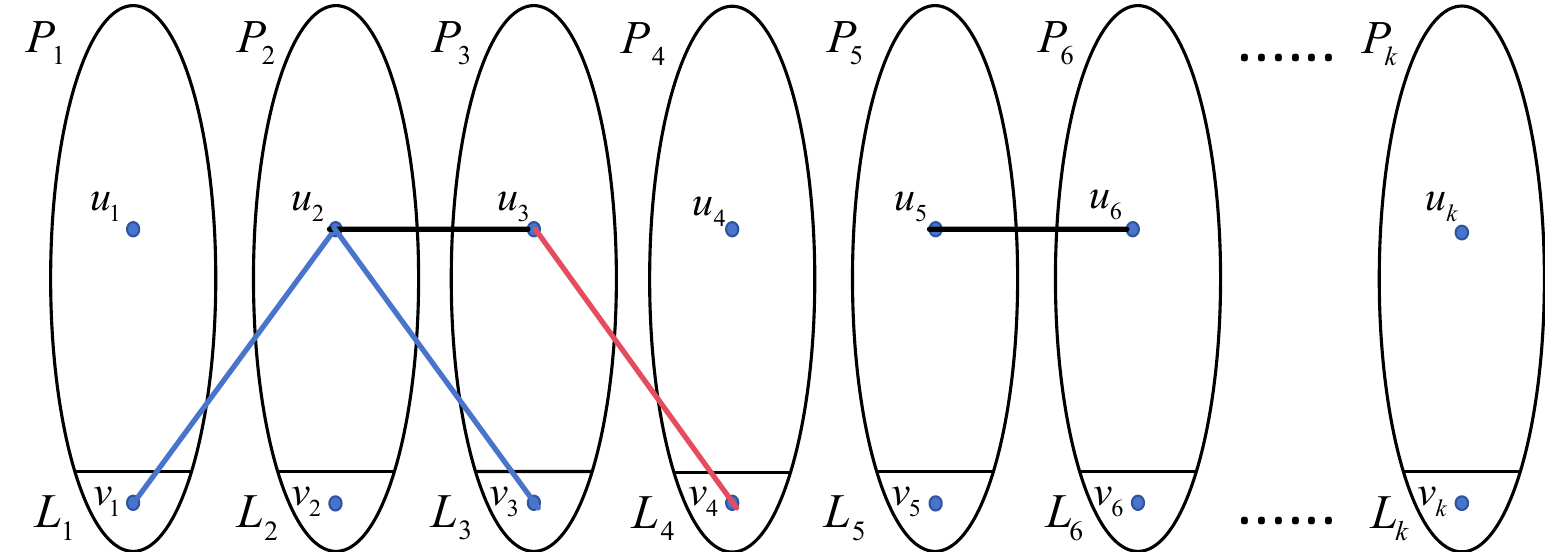}
    \caption{ $M'=u_2u_3\cup u_5u_6\cup\{u_1,u_4,u_7\ldots,u_k\}$}
    \label{fig0126}
\end{figure*}
Then we obtain a new $\{P^*_3, M_2\}$-tiling $\mathcal{H}'$ by replacing $M$ with $v_1 u_2 v_3 \cup \{u_4, v_5, v_6, u_7, \dots, u_k\}$ and $u_3 v_4 \cup u_5 u_6 \cup \{u_1, v_2, v_7, \dots, v_k\}$, where each $v_s \in L_s$. Clearly, $|V(\mathcal{H}')| = |V(\mathcal{H})| + k$ contradicts the maximality of $\mathcal{H}$.
\end{proof}

\begin{claim}\label{2p3}
There is at most one non-isolated vertex $u_i$ in $P^*$ or $M$ that satisfies $\overleftarrow{e}(u_i, L) > 0$ and $\overrightarrow{e}(u_i, L) > 0$.
\end{claim}
\begin{proof}
Suppose that there are two non-isolated vertices $u_{i_1}$ and $u_{i_2}$ in $P^*$ or $M$ such that $\overleftarrow{e}(u_{i_j}, L) > 0$ and $\overrightarrow{e}(u_{i_j}, L) > 0$ for $j\in [2].$ Assume $i_1 < i_2$. 
We consider two cases separately.

\textbf{Case 1.} $i_2-i_1\neq 2$.
Take $v_{j_1}\in N(u_{i_1})\cap L_{i_1-1}$, $v_{j_2}\in N(u_{i_1})\cap L_{i_1+1}$ and $v_{l_1}\in N(u_{i_2})\cap L_{i_2-1}$, $v_{l_2}\in N(u_{i_2})\cap L_{i_2+1}$.
Then we obtain a new $\{P^*_3,M_2\}$-tiling $\mathcal{H}'$ by replacing $H$ with $v_{j_1}u_{l_1}v_{j_2}\cup(\{u_1,\cdots, u_k\}\setminus\{u_{l_1-1},u_{l_1},u_{l_1+1}\})$ and $v_{q_1}u_{l_2}v_{q_2}\cup(\{v_1,\cdots, v_k\}\setminus\{v_{l_2-1},v_{l_2},v_{l_2+1}\})$, where $v_s\in L_s$. It is clear that $|V(\mathcal{H}')|=|V(\mathcal{H})|+k$ contradicts the maximality of $\mathcal{H}$.

{\bf Case 2.} $i_2-i_1=2$.
It is easy to verify that this contradicts the Claims \ref{P^*}-\ref{M}.
\end{proof}

From Claim \ref{2p3}, the following Claim can be directly obtained.

\begin{claim}\label{more5}
$e(I_{P^*},L)\le 4|L|/k$.    
\end{claim}

\begin{claim}\label{more5M}
$e(I_M,L)\le 4|L|/k$ and $e(I_{M'},L)\le 4|L|/k$. 
\end{claim}
\begin{proof}
Suppose that $e(I_M,L)> 4|L|/k$ and $e(I_{M'},L)> 4|L|/k$. Then there are at least three non-isolated vertices of $M$ are adjacent to $L$. 
By Claim \ref{2p3}, we have at most one non-isolated vertex $u_i$ of $M$ satisfying $\overleftarrow e(u_i,L)>0$ and $\overrightarrow e(u_i,L)>0$.
We distinguish between two cases.

\textbf{Case 1.} All four non-isolated vertices are adjacent to $L$. By symmetry, we only need to consider $i = 2$ or $i = 3$.

{\bf Case 1.1.} $\overleftarrow e(u_2,L)>0$ and $\overrightarrow e(u_2,L)>0$. 

By Claim~\ref{M'} (1) and (3), we have $e(u_3,L)=\overleftarrow e(u_5,L)=\overrightarrow e(u_6,L)=0$, which contradicts $e(I_{M'},L)> 4|L|/k$. 

By Claim~\ref{M} (1) and (3), we have $\overleftarrow{e}(u_3, L) = \overleftarrow{e}(u_4, L) = \overrightarrow{e}(u_5, L) = 0$.
Hence, $\overrightarrow{e}(u_3, L) > 0$, $\overrightarrow{e}(u_4, L) > 0$, and $\overleftarrow{e}(u_5, L) > 0$.
Since $e(I_{M}, L) > 4|L|/k$, we have $N(u_3) \cap N(u_5) \cap L_4 \neq \emptyset$. Pick $v_1 \in N(u_2) \cap L_1$, $v_3 \in N(u_2) \cap L_3$, $v_4 \in N(u_3) \cap N(u_5) \cap L_4$, and $v_5 \in N(u_4) \cap L_5$. Then we obtain a new $\{P^*_3, M_2\}$-tiling $\mathcal{H}'$ by replacing $M$ with $v_1 u_2 v_3 \cup \{u_1, u_4, v_5, u_6, \dots, u_k\}$ and $u_3 v_4 u_5 \cup \{v_2, v_6, \dots, v_k\}$, where $v_s \in L_s$. Clearly, $|V(\mathcal{H}')| = |V(\mathcal{H})| + k$ contradicts the maximality of $\mathcal{H}$.

{\bf Case 1.2.} $\overleftarrow{e}(u_3, L) > 0$ and $\overrightarrow{e}(u_3, L) > 0$. By Claim~\ref{M} (1) and (2), we have $e(u_2, L) = \overleftarrow{e}(u_4, L) = \overrightarrow{e}(u_5, L) = 0$, which contradicts $e(I_M, L) > 4|L|/k$. (The case for $M'$ is similar.)

{\bf Case 2.} There are three non-isolated vertices adjacent to $L$. Then there exist two non-isolated vertices $u_{i_1}$ and $u_{i_2}$ such that $\overleftarrow{e}(u_{i_j}, L) > 0$ and $\overrightarrow{e}(u_{i_j}, L) > 0$ for $j \in [2]$. By Claim~\ref{2p3}, we only need to consider the case where $u_{i_1}$ and $u_{i_2}$ are two non-isolated vertices of $M'$. Assume $i_1 < i_2$ and $e(u_{i_j}, L) > 0$ for $j \in [2]$. We consider two subcases.

{\bf Case 2.1.} $i_2-i_1\neq 2$.
Take $v_{j_1}\in N(u_{i_1})\cap L_{i_1-1}$, $v_{j_2}\in N(u_{i_1})\cap L_{i_1+1}$, $v_{l_1}\in N(u_{i_2})\cap L_{i_2-1}$ and $v_{l_2}\in N(u_{i_2})\cap L_{i_2+1}$. Then we obtain a new $\{P^*_3,M_2\}$-tiling $\mathcal{H}'$ by replacing $H$ with $v_{j_1}u_{i_1}v_{j_2}\cup(\{u_1,\cdots, u_k\}\setminus\{u_{i_1-1},u_{i_1},u_{i_1+1}\})$ and $v_{l_1}u_{i_2}v_{l_2}\cup(\{v_1,\cdots, v_k\}\setminus\{v_{i_2-1},v_{i_2},v_{i_2+1}\})$, where $v_s\in L_s$. Clearly, $|V(\mathcal{H}')| = |V(\mathcal{H})| + k$ contradicts the maximality of $\mathcal{H}$.

{\bf Case 2.2.} $i_2 - i_1 = 2$. Recall that $M' = u_2 u_3 \cup u_5 u_6 \cup \{u_1, u_4, u_7, \ldots, u_k\}$. Since $u_{i_j}$ for $j \in [2]$ are non-isolated vertices of $M'$, we have $i_1 = 3$ and $i_2 = 5$. Choose $v_2 \in N(u_3) \cap L_2$, $v_4 \in N(u_3) \cap L_4$, $v'_4 \in N(u_5) \cap L_4$, and $v_6 \in N(u_5) \cap L_6$; see Figure~\ref{fig13}.
Moreover, at least one of $e(u_2, L) > 0$ and $e(u_6, L) > 0$ holds. 
By symmetry, without loss of generality, assume that $e(u_2, L) > 0$. 
Take $v_j \in N(u_2) \cap L$.
We can obtain a new $\{P^*_3, M_2\}$-tiling $\mathcal{H}''$ by replacing $M'$ with $v_2 u_3 v_4 \cup \{u_1, v_5, v_6, u_7, \dots, u_k\}$ and $v_j u_2 \cup u_5 u_6 \cup \bigl( \{v_1, v_3, u_4, v_7, \dots, v_k\} \setminus \{j\} \bigr)$, where $v_s \in L_s$. Clearly, $|V(\mathcal{H}'')| = |V(\mathcal{H})| + k$ contradicts the maximality of $\mathcal{H}$.
\begin{figure*}[htbp]
    \centering
    \includegraphics[trim=4cm 5.7cm 3.7cm 4cm,clip, scale=0.5]{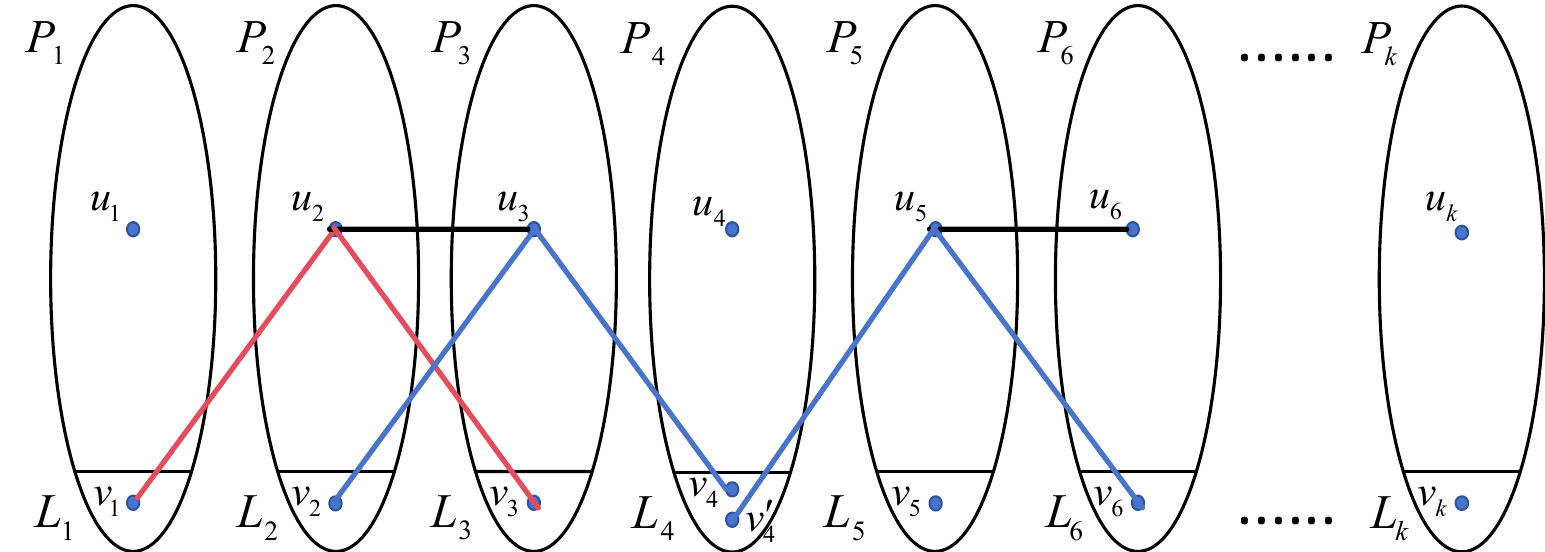}
    \caption{$M'=u_2u_3\cup u_5u_6\cup\{u_1,u_4,u_7\ldots,u_k\}$}
    \label{fig13}
\end{figure*}
\end{proof}

\begin{claim}\label{4edges}
If $e(I_{P^*},L)> 3|L|/k$, then $e(J_{P^*},L)\le 4|L|/k-e(I_{P^*},L)$.    
\end{claim}
\begin{proof}
Suppose that $e(J_{P^*}, L) > 4|L|/k - e(I_{P^*}, L)$. 
By Claim~\ref{2p3}, there is at most one non-isolated vertex $u_i$ of $P^*$ satisfying $\overleftarrow{e}(u_i, L) > 0$ and $\overrightarrow{e}(u_i, L) > 0$. 
Recall that $P^* = u_2 u_3 u_4 \cup \{u_1, u_5, \dots, u_k\}$. 
Since $e(I_{P^*}, L) > 3|L|/k$, by symmetry, we only need to consider the case where $e(u_i, L) > 0$ for each $i \in \{2,3\}$.

{\bf Case 1.} $\overleftarrow e(u_2,L)>0$ and $\overrightarrow e(u_2,L)>0$. By Claim~\ref{P^*} (3), we have $\overleftarrow{e}(u_3, L) = \overrightarrow{e}(u_4, L) = 0$. Therefore, $\overrightarrow{e}(u_3, L) > 0$ and $\overleftarrow{e}(u_4, L) > 0$. Since $e(I_{P^*}, L) > 3|L|/k$, we have $N(u_2) \cap N(u_4) \cap L_3 \neq \emptyset$. Take $v_1 \in N(u_2) \cap L_1$, $v_3 \in N(u_2) \cap N(u_4) \cap L_3$ and $v_4 \in N(u_3) \cap L_4$; see Figure~\ref{fig14}.

Suppose that there exists an isolated vertex $u_i$ of $P^*$ such that $e(u_i, L) > 0$. Choose $v_j \in N(u_i) \cap L$. We can obtain a new $\{P^*_3, M_2\}$-tiling $\mathcal{H}'$ by replacing $P^*$ with 
$u_2 v_3 u_4 \cup \bigl( \{v_i, u_1, u_5, \dots, u_k\} \setminus \{u_i\} \bigr)$ and $u_i v_j \cup u_3 v_4 \cup \bigl( \{v_1, \dots, v_k\} \setminus \{v_i, v_j, v_3, v_4\} \bigr),$
where $v_s \in L_s$. Note that this would fail only when $i=5$ and $j=4$ and we may further assume that $u_5$ is the only isolated vertex of $P^*$ satisfying $e(u_5, L) > 0$. In this case, since $e(I_{P^*}, L) + e(J_{P^*}, L) > 4|L|/k$, we can choose $v_4 \in N(u_3) \cap N(u_5) \cap L_4$. Clearly, we can enlarge $\mathcal{H}$ by replacing $P^*$ with two disjoint copies of $P_3^*$, a contradiction.

\begin{figure*}[htbp]
    \centering
    \includegraphics[trim=4cm 5.7cm 4cm 4.2cm,clip, scale=0.5]{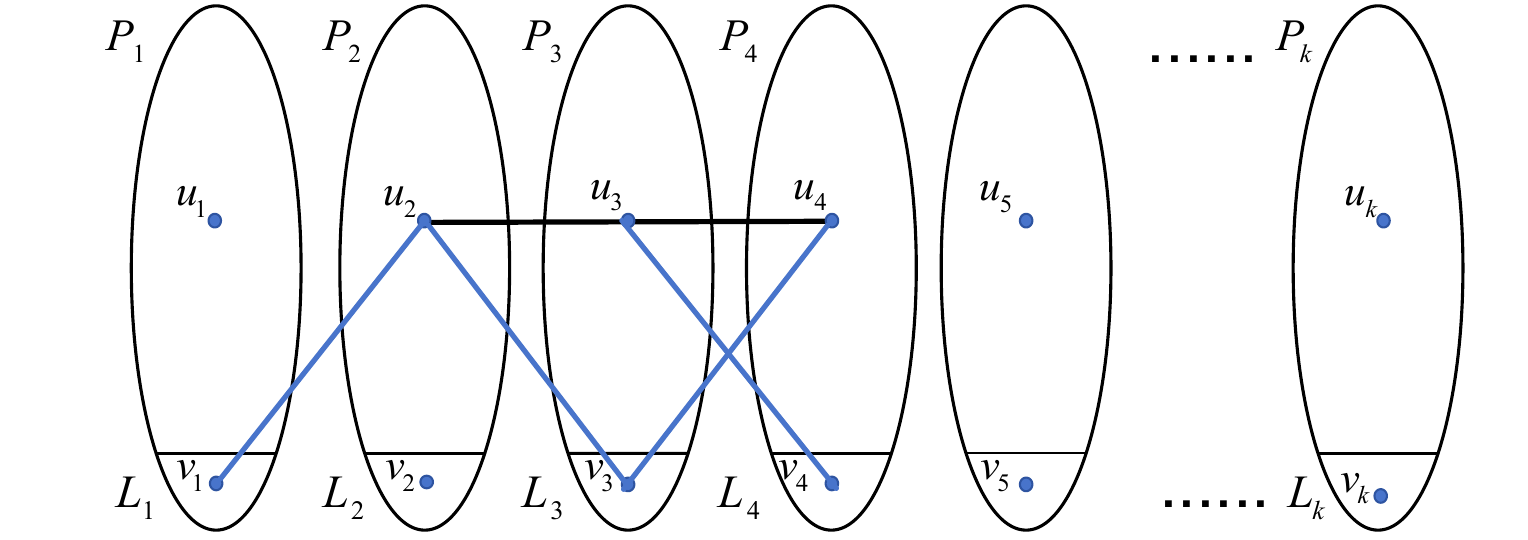}
    \caption{ $P^*=u_2u_3u_4\cup\{u_1,u_5,\cdots, u_k\}$}
    \label{fig14}
\end{figure*}

{\bf Case 2.} $\overleftarrow e(u_3,L)>0$ and $\overrightarrow e(u_3,L)>0$. By Claim~\ref{P^*} (2), we have at most one of $\overleftarrow{e}(u_2, L) > 0$, $\overrightarrow{e}(u_2, L) > 0$, $\overleftarrow{e}(u_4, L) > 0$ and $\overrightarrow{e}(u_4, L) > 0$ holds. This contradicts $e(I_{P^*}, L) > 3|L|/k$.
\end{proof}

\begin{claim}\label{4edgesM}
    If $e(I_M,L)>3|L|/k$ and $e(I_{M'},L)>3|L|/k$, then $e(J_M,L)\le 4|L|/k-e(I_M,L)$ and $e(J_{M'},L)\le 4|L|/k-e(I_{M'},L)$.  
\end{claim}

\begin{proof}
Suppose that $e(J_M, L) > 4|L|/k - e(I_M, L)$ and $e(J_{M'}, L) > 4|L|/k - e(I_{M'}, L)$. 
By Claim~\ref{2p3}, we have at most one non-isolated vertex $u_i$ of $M$ satisfies $\overleftarrow{e}(u_i, L) > 0$ and $\overrightarrow{e}(u_i, L) > 0$. 
Since $e(I_M, L) > 3|L|/k$ and $e(I_{M'}, L) > 3|L|/k$, we only need to consider three possible cases.

{\bf Case 1.} $e(u_i, L) > 0$ for every non-isolated vertex $u_i$ of $M$ or $M'$.

{\bf Case 1.1.} $\overleftarrow{e}(u_2, L)>0$. Pick $v_1\in N(u_2)\cap L_1$.

By Claim~\ref{M}~(1), we can obtain that $\overleftarrow{e}(u_4,L)=\overrightarrow{e}(u_5,L)=0$ for $M$. Hence, $\overrightarrow{e}(u_4,L)>0,\overleftarrow{e}(u_5,L)>0$. Choose $v_4\in N(u_5)\cap L_4$ and $v_5\in N(u_4)\cap L_5$.
We can obtain a new $\{P^*_3,M_2\}$-tiling $\mathcal{H}'$ by replacing $M$ with $v_1u_2\cup\ u_4v_5\cup\ \{v_3,u_6,\cdots, u_k\}$ and $v_ju_3\cup v_4u_5\cup(\{u_1,v_2,v_6,\cdots, v_k\}\setminus\{u_j\})$, where $v_s\in L_s$. Note that if $\overrightarrow{e}(u_3,L)>0$, then we can obtain $v_4\in N(u_3)\cap N(u_i)\cap L_4$ since $e(I_M,L)> 3|L|/k$. Now, $v_ju_3\cup v_4u_5\cup(\{u_1,v_2,v_6,\cdots, v_k\}\setminus\{u_j\})$ is a copy of $P^*_3$. This contradicts the maximality of $\mathcal{H}$.

By Claim \ref{M'}~(1), we can obtain that $\overleftarrow{e}(u_5,L)=\overrightarrow{e}(u_6,L)=0$ for $M'$. 
Hence, $\overrightarrow{e}(u_5,L)>0,\overleftarrow{e}(u_6,L)>0$. Choose $v_5\in N(u_6)\cap L_5$ and $v_6\in N(u_5)\cap L_6$.
Since $e(u_3,L)>0$, we can take $v_j\in N(u_3)\cap L$. 
We can obtain a new $\{P^*_3,M_2\}$-tiling $\mathcal{H}'$ by replacing $M'$ with $v_1u_2\cup\ u_5v_6\cup\ \{v_3,u_4,u_7,\cdots, u_k\}$ and $v_ju_3\cup v_5u_6\cup(\{u_1,v_2,v_4,v_7,\cdots, v_k\}\setminus\{u_j\})$, where $v_s\in L_s$. This contradicts the maximality of $\mathcal{H}$.
\begin{figure*}[htbp]
    \centering
    \includegraphics[trim=4cm 5.7cm 3.7cm 4cm,clip, scale=0.5]{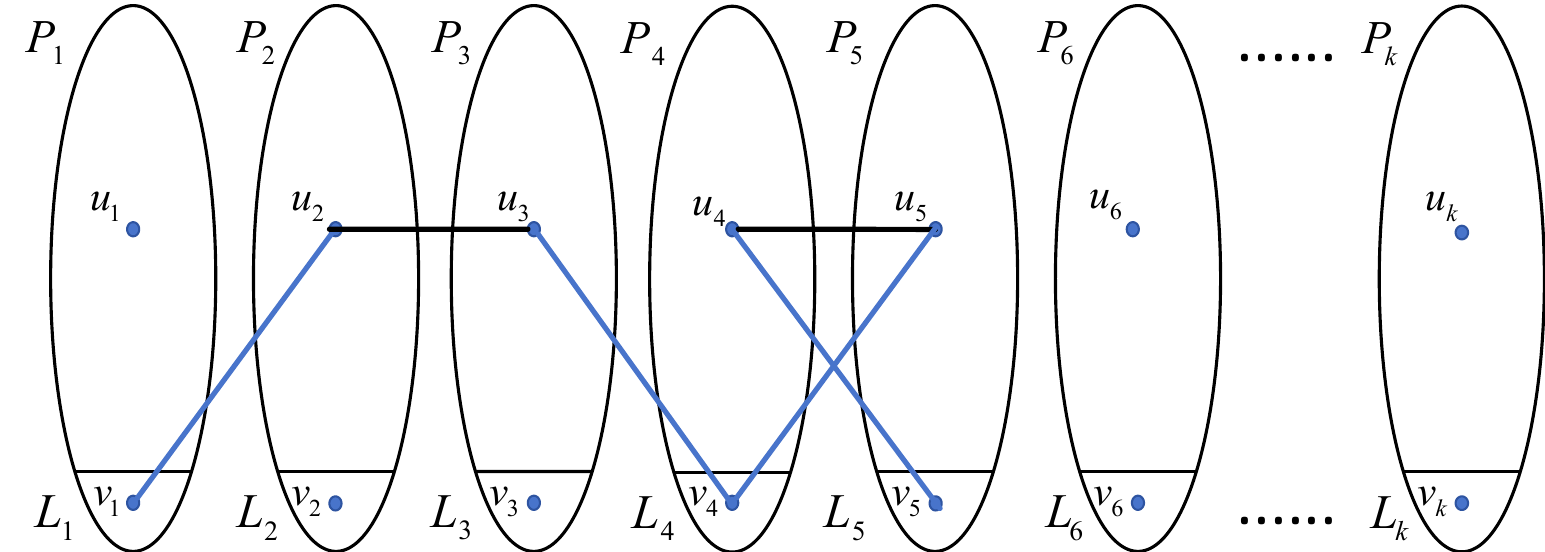}
    \caption{ $M=u_2u_3\cup u_4u_5\cup\{u_1,u_6,u_7\ldots,u_k\}$}
    \label{fig16}
\end{figure*}

{\bf Case 1.2.} $\overrightarrow{e}(u_2, L)>0$. 

If $\overrightarrow{e}(u_3, L) > 0$, then by Claim~\ref{M}~(1), we have $\overleftarrow{e}(u_4, L) = \overrightarrow{e}(u_5, L) = 0$ for $M$. 
Thus, we have $\overrightarrow{e}(u_4, L) > 0$ and $\overleftarrow{e}(u_5, L) > 0$. 
Since $e(I_M, L) > 3|L|/k$, we obtain $N(u_3) \cap N(u_5) \cap L_4 \neq \emptyset$.
Pick $v_3 \in N(u_2) \cap L_3$, $v_4 \in N(u_3) \cap N(u_5) \cap L_4$, and $v_5 \in N(u_4) \cap L_5$; see Figure~\ref{fig17}.
Then we can obtain a new $\{P^*_3,M_2\}$-tiling $\mathcal{H}''$  by replacing $M$ with $u_3v_4u_5\cup\ \{u_1,v_2,u_6,\cdots, u_k\}$ and $u_2v_3\cup u_4v_5\cup\{v_1,v_6,\cdots, v_k\}$, where $v_s\in L_s$. This contradicts the maximality of $\mathcal{H}$.

By Claim~\ref{M'}~(1), we have $\overrightarrow{e}(u_6, L) = 0$ for $M'$.
Hence, we have $\overleftarrow{e}(u_6, L) > 0$. Pick $v_3 \in N(u_2) \cap L_3$ and $v_5 \in N(u_4) \cap L_5$. If $\overleftarrow{e}(u_5, L) > 0$, then we have $N(u_3) \cap N(u_5) \cap L_4 \neq \emptyset$ since $e(I_M, L) > 3|L|/k$. Take $v_4 \in N(u_3) \cap N(u_5) \cap L_4$. If $\overrightarrow{e}(u_5, L) > 0$, pick $v_6 \in N(u_5) \cap L_6$.
Then we can obtain a new $\{P^*_3,M_2\}$-tiling $\mathcal{H}''$  either by replacing $M'$ with $u_2v_3\cup u_5v_6\cup\{u_1,u_4,u_7,\cdots, u_k\}$ and $u_3v_4\cup v_5u_6\cup\{v_1,v_2,v_7,\cdots, v_k\}$ or  by replacing $M'$ with $u_2v_3\cup u_5u_6\cup\{u_1,u_4,u_7,\cdots, u_k\}$ and $u_3v_4u_5\cup\{v_1,v_2,v_7,\cdots, v_k\}$, where $v_s\in L_s$. This contradicts the maximality of $\mathcal{H}$.
\begin{figure*}[htbp]
    \centering
    \includegraphics[trim=4cm 5.7cm 3.7cm 4cm,clip, scale=0.5]{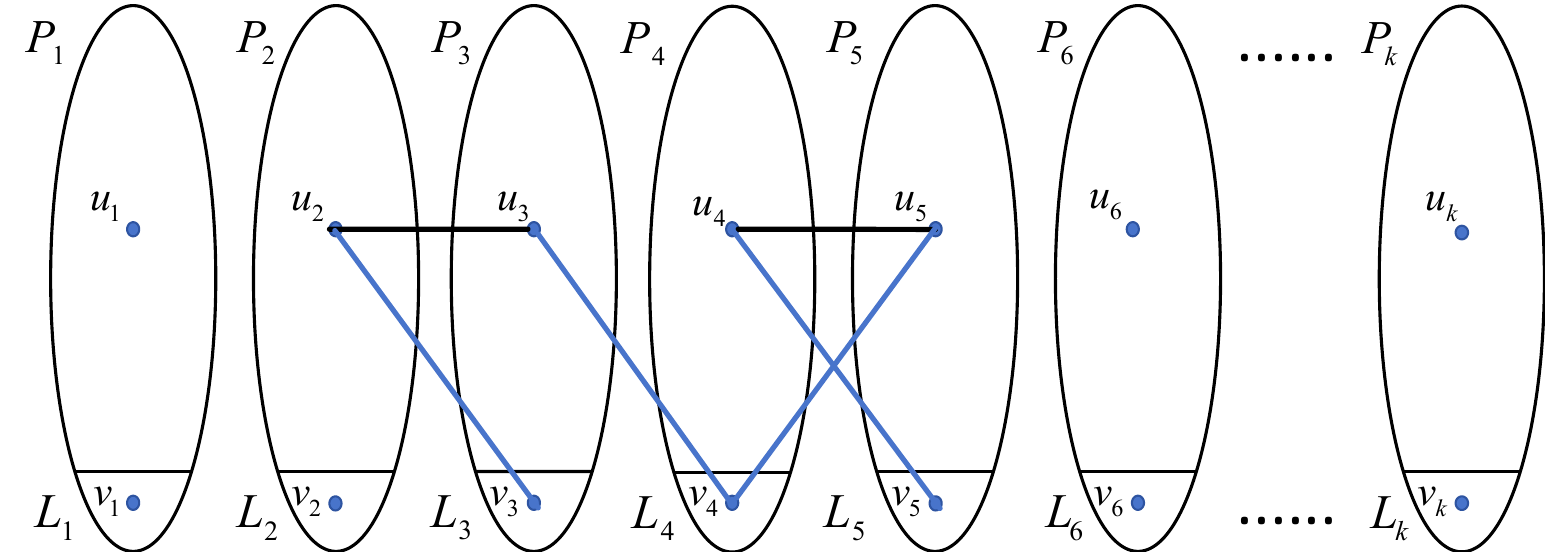}
    \caption{ $M=u_2u_3\cup u_4u_5\cup\{u_1,u_6,u_7\ldots,u_k\}$}
    \label{fig17}
\end{figure*}
  
If $\overleftarrow{e}(u_3, L)>0$, recall that $\overrightarrow{e}(u_2, L)>0$. By the symmetry of $M$, we only need to consider $\overrightarrow{e}(u_4,L)>0$ and $\overleftarrow{e}(u_5,L)>0$. Take $v_2\in N(u_3)\cap L_2$, $v_3\in N(u_2)\cap L_3$, $v_4\in N(u_5)\cap L_4$, $v_5\in N(u_4)\cap L_5$; see Figure~\ref{fig18}.
We can obtain a new $\{P^*_3,M_2\}$-tiling $\mathcal{H}'''$ by replacing $M$ with $u_2v_3\cup\ u_4v_5\cup\ \{u_1,u_6,\cdots, u_k\}$ and $v_2u_3\cup v_4u_5\cup\{v_1,v_6,\cdots, v_k\}$, where $v_s\in L_s$. Clearly, this contradicts the maximality of $\mathcal{H}$.

Similarly, by the symmetry of $M'$, we only need to consider $\overrightarrow{e}(u_5,L)>0$ and $\overleftarrow{e}(u_6,L)>0$. Take $v_2\in N(u_3)\cap L_2$, $v_3\in N(u_2)\cap L_3$, $v_5\in N(u_6)\cap L_5$ and $v_6\in N(u_5)\cap L_6$.
We can obtain a new $\{P^*_3,M_2\}$-tiling $\mathcal{H}'''$  by replacing $M'$ with $u_2v_3\cup\ u_5v_6\cup\ \{u_1,u_4,u_7,\cdots, u_k\}$ and $v_2u_3\cup v_5u_6\cup\{v_1,v_4,v_7,\cdots, v_k\}$, where $v_s\in L_s$. Clearly, this contradicts the maximality of $\mathcal{H}$.
\begin{figure*}[htbp]
    \centering
    \includegraphics[trim=4cm 5.7cm 3.7cm 4cm,clip, scale=0.5]{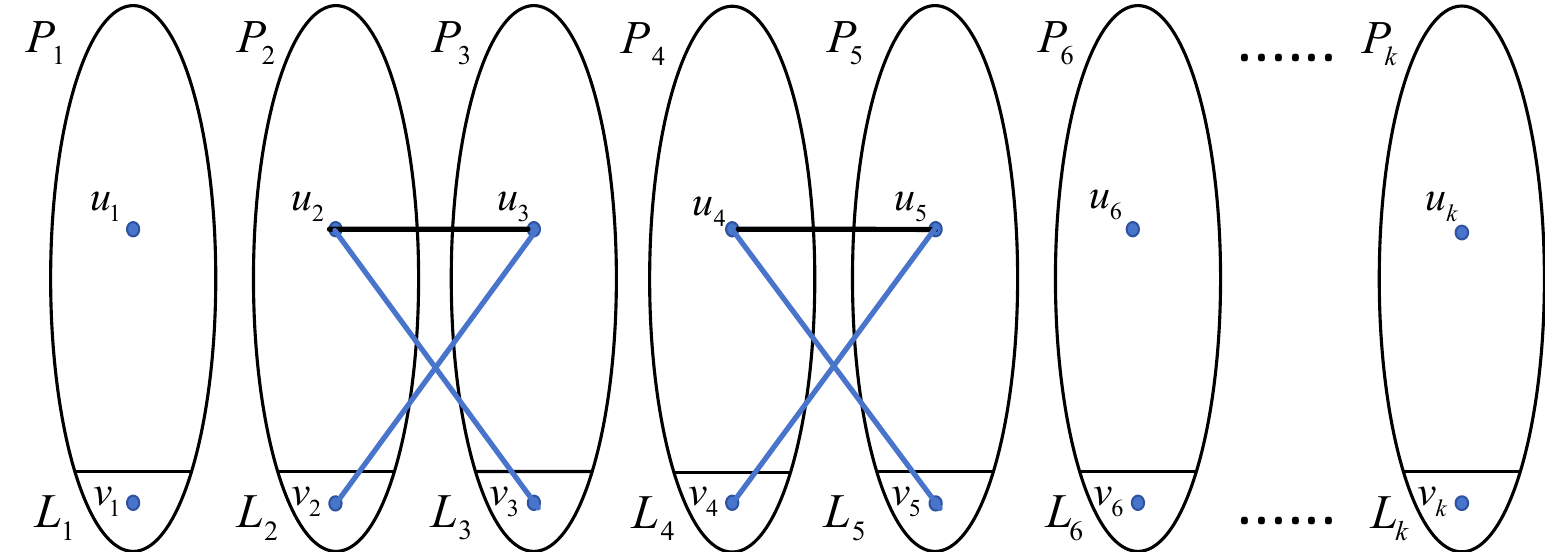}
    \caption{ $M=u_2u_3\cup u_4u_5\cup\{u_1,u_6,u_7\ldots,u_k\}$}
    \label{fig18}
\end{figure*}

{\bf Case 2.} There exist exactly three non-isolated vertices $u_{i_1}, u_{i_2}, u_{i_3}$ of $M$ or $M'$ such that $e(u_{i_j}, L) > 0$ for every $j \in [3]$. Since $e(I_M, L) > 3|L|/k$, there exists a non-isolated vertex $u_{i_l}$ with $\overleftarrow{e}(u_{i_l}, L) > 0$ and $\overrightarrow{e}(u_{i_l}, L) > 0$ for some $l \in [3]$. 
By the symmetry of $M$ and $M'$, we only need to consider $i_l = 2$ or $i_l = 3$.

{\bf Case 2.1.} $\overleftarrow{e}(u_2,L)>0$ and $\overrightarrow{e}(u_2,L)>0$. Choose $v_1\in N(u_2)\cap L_1$, $v_3\in N(u_2)\cap L_3$. 

By Claim~\ref{M'} (1) and (3), we have $\overleftarrow{e}(u_5, L) = \overrightarrow{e}(u_6, L) = e(u_3, L) =0$ for $M'$.
Thus, we only need to consider $\overrightarrow{e}(u_5, L)>0$ and $\overleftarrow{e}(u_6, L)>0$ for $M'$.
Take $v_5\in N(u_6)\cap L_5$, $v_6\in N(u_5)\cap L_6$.
Suppose there exists an isolated vertex $u_i$ such that $e(u_i,L)>0$.  
Choose $v_j\in N(u_i)\cap L$.
If $i=4$ and $j=5$, we may further assume that $u_4$ is the only isolated vertex of $M'$ satisfying $e(u_4, L) > 0$. 
Since $e(I_{M'},L)+e(J_{M'},L)> 4|L|/k$, we can choose $v_5 \in N(u_4) \cap N(u_6) \cap L_5$.
Then we can obtain a new $\{P^*_3,M_2\}$-tiling $\mathcal{H}'$  by replacing $M'$ with $v_1u_2u_3\cup\{v_4,u_5,v_6,u_7\cdots, u_k\}$ and $u_4v_5u_6\cup\{u_1,v_2,v_3,v_7,\cdots, v_k\}$, where $v_s\in L_s$, a contradiction.
If $i=7$ and $j=6$, we may further assume that $u_7$ is the only isolated vertex of $M'$ satisfying $e(u_7, L) > 0$. 
Since $e(I_{M'},L)+e(J_{M'},L)> 4|L|/k$, we can choose $v_6 \in N(u_5) \cap N(u_7) \cap L_6$.
Then, similarly, we can enlarge $\mathcal{H}$ by replacing $M'$ with two disjoint copies of $P^*_3$, a contradiction.
For other cases,
we can obtain a new $\{P^*_3,M_2\}$-tiling $\mathcal{H}'$ by replacing $M'$ with $v_1u_2u_3\cup(\{v_i,v_4,u_5,v_6,u_7\cdots, u_k\}\setminus\{u_i\})$ and $u_iv_j\cup v_5u_6\cup(\{u_1,v_2,v_3,u_4,v_7,\cdots, v_k\}\setminus\{v_i,v_j\})$, where $v_s\in L_s$. 
Note that this would fail only when $i=k$ and $j=2$, then we may further assume that $u_k$ is the only isolated vertex of $M'$ satisfying $e(u_k, L) > 0$.
In this case, since $e(I_{M'},L)+e(J_{M'},L)> 4|L|/k$, we can choose $v_1 \in N(u_k) \cap N(u_2) \cap L_1$. Clearly, we can enlarge $\mathcal{H}$ by replacing $M'$ with two disjoint copies of $M_2$, a contradiction.

By Claim~\ref{M} (1) and (3), we have $\overleftarrow{e}(u_4, L) = \overrightarrow{e}(u_5, L) = \overleftarrow{e}(u_3, L) =0$ for $M$.
We only need to consider three cases for $M$.

{\bf Case 2.1.1.} $\overrightarrow{e}(u_3, L)>0$ and $ \overrightarrow{e}(u_4, L)>0$ for $M$.
Take $v_4\in N(u_3)\cap L_4$ and $v_5\in N(u_4)\cap L_5$.
Suppose there exists an isolated vertex $u_i$ such that $e(u_i,L)>0$. 
Choose $v_j\in N(u_i)\cap L$ and we can obtain a new $\{P^*_3,M_2\}$-tiling $\mathcal{H}'$ by replacing $M$ with $v_1u_2v_3\cup\{v_i,u_4,u_5,\cdots, u_k\}\setminus\{u_i\}$ and $u_iv_j\cup u_3v_4\cup(\{v_1,\cdots, v_k\}\setminus\{v_i,v_j,v_3,v_4\})$, where $v_s\in L_s$. Clearly, $|V(\mathcal{H}')|=|V(\mathcal{H})|+k$ contradicts the maximality of $\mathcal{H}$.

{\bf Case 2.1.2.}  $\overrightarrow{e}(u_3, L)>0$ and $\overleftarrow{e}(u_5, L)>0$ for $M$. Since $e(I_M,L)>3|L|/k$, we have  $N(u_3)\cap N(u_5)\cap L_4\neq\emptyset$.
Take $v_4\in N(u_3)\cap N(u_5)\cap L_4$.  Then we can obtain a new $\{P^*_3,M_2\}$-tiling $\mathcal{H}''$ by replacing $M$ with $v_1u_2v_3\cup\{u_4,v_5,u_6,\cdots, u_k\}$ and $u_3v_4u_5\cup\{u_1,v_2,v_6,\cdots, v_k\}$, where $v_s\in L_s$.  Clearly, $|V(\mathcal{H}'')|=|V(\mathcal{H})|+k$ contradicts the maximality of $\mathcal{H}$.

{\bf Case 2.1.3.} $\overrightarrow{e}(u_4, L)>0$ and $\overleftarrow{e}(u_5, L)>0$ for $M$.
Take $v_4\in N(u_5)\cap L_4$, $v_5\in N(u_4)\cap L_5$.
Suppose there exists an isolated vertex $u_i$ such that $e(u_i,L)>0$.  
Choose $v_j\in N(u_i)\cap L$ and we can obtain a new $\{P^*_3,M_2\}$-tiling $\mathcal{H}'''$ by replacing $M$ with $v_1u_2v_3\cup(\{v_i,v_4,u_5,\cdots, u_k\}\setminus\{u_i\})$ and $u_iv_j\cup u_4v_5\cup(\{u_1,v_2,u_3,v_6,\cdots, v_k\}\setminus\{v_i,v_j\})$, where $v_s\in L_s$. This contradicts the maximality of $\mathcal{H}$.
Note that this would fail only when $i=6$ and $j=5$, then we may further assume that $u_6$ is the only isolated vertex of $M$ satisfying $e(u_6, L) > 0$.
In this case, since $e(I_{M},L)+e(J_{M},L)> 4|L|/k$, we can choose $v_5\in N(u_6)\cap N(u_6)\cap L_5$. Clearly, we can enlarge $\mathcal{H}$ by replacing $M$ with two disjoint copies of $P_3^*$, a contradiction.

{\bf Case 2.2.} $\overleftarrow{e}(u_3,L)>0$ and $\overrightarrow{e}(u_3,L)>0$. 
Choose $v_2\in N(u_3)\cap L_2$ and $v_4\in N(u_3)\cap L_4$.

By Claim~\ref{M} (1) and (2), we have $\overleftarrow{e}(u_4, L) = \overrightarrow{e}(u_5, L) = e(u_2, L) =0$ for $M$. Then we have $\overrightarrow{e}(u_4, L)>0$ and $\overleftarrow{e}(u_5, L)>0$.
Pick $v_4\in N(u_5)\cap L_4$ and $v_5\in N(u_4)\cap L_5$.
Suppose there exists an isolated vertex $u_i$ such that $e(u_i,L)>0$.  
Choose $v_j\in N(u_i)\cap L$.
Then we can obtain a new $\{P^*_3,M_2\}$-tiling $\mathcal{H}'$ by replacing $M$ with $v_2u_3v_4\cup(\{v_i,u_1,u_5,\cdots, u_k\}\setminus\{u_i\})$ and $u_iv_j\cup u_4v_5\cup(\{v_1,u_2,v_3,v_6,\cdots, v_k\}\setminus\{v_i,v_j\})$, where $v_s\in L_s$. This contradicts the maximality of $\mathcal{H}$.
Note that this would fail only when $i=6$ and $j=5$, then we may further assume that $u_6$ is the only isolated vertex of $M$ satisfying $e(u_6, L) > 0$.
In this case, since $e(I_{M},L)+e(J_{M},L)> 4|L|/k$, we can choose $v_5\in N(u_4)\cap N(u_6)\cap L_5$. Clearly, we can enlarge $\mathcal{H}$ by replacing $M$ with two disjoint copies of $P_3^*$, a contradiction.

By Claim~\ref{M'} (1) and (2), we have $\overrightarrow{e}(u_6, L) = e(u_2, L) =0$ for $M'$. 
Then we have $\overleftarrow{e}(u_6, L)>0$.
Pick $v_5\in N(u_6)\cap L_5$.
Suppose there exists an isolated vertex $u_i$ such that $e(u_i,L)>0$.  
Choose $v_j\in N(u_i)\cap L$.
Since $e(I_{M'},L)>3|L|/k$, we have $e(u_5,L)>0$.
If $i\neq 7$,
then we can obtain a new $\{P^*_3,M_2\}$-tiling $\mathcal{H}'$ by replacing $M'$ with $u_2u_3v_4\cup(\{v_i,u_1,v_5,u_6,\cdots, u_k\}\setminus\{u_i\})$ and $u_iv_j\cup u_5v_6\cup(\{v_1,v_2,v_3,u_4,v_7,\cdots, v_k\}\setminus\{v_i,v_j\})$, where $v_s\in L_s$. Clearly, $|V(\mathcal{H}')|=|V(\mathcal{H})|+k$ contradicts the maximality of $\mathcal{H}$.
Note that this would fail only when $i=4$ and $j=5$, then we may further assume that $u_4$ is the only isolated vertex of $M'$ satisfying $e(u_4, L) > 0$.
In this case, since $e(I_{M'},L)+e(J_{M'},L)> 4|L|/k$, we can choose $v_5\in N(u_4)\cap N(u_6)\cap L_5$. Clearly, we can enlarge $\mathcal{H}$ by replacing $M'$ with two disjoint copies of $P_3^*$, a contradiction.
If $i=7$, take $v_l\in N(u_5)\cap L$.
Then we can obtain a new $\{P^*_3,M_2\}$-tiling $\mathcal{H}'$  by replacing $M'$ with $u_1u_2\cup u_6v_5\cup\{u_3,u_4,v_7,u_8,\cdots, u_k\}$ and $u_7v_j\cup u_5v_l\cup(\{v_1,\cdots, v_k\}\setminus\{v_7,v_j,v_5,v_l\})$, where $v_s\in L_s$. Those contradict the maximality of $\mathcal{H}$.
Note that this would fail only when  $j=6$ and $l=6$, then we may further assume that $u_7$ is the only isolated vertex of $M'$ satisfying $e(u_7, L) > 0$.
In this case, since $e(I_{M'},L)+e(J_{M'},L)> 4|L|/k$, we can choose $v_6\in N(u_7)\cap N(u_5)\cap L_6$. Clearly, we can enlarge $\mathcal{H}$ by replacing $M'$ with two disjoint copies of $P_3^*$, a contradiction.

{\bf Case 3.} There are two non-isolated vertices are adjacent to $L$. 
By Claim \ref{2p3}, we only need to consider the case for $M'$.
This case coincides with the proof of Case 2 of Claim~\ref{more5M}, so we omit the details here.
\end{proof}

\begin{claim}\label{iso}
Let $u_l\in V$ be an isolated vertex of any copy of $P^*_3$ or $M_2$ in $\mathcal{H}$. Then at most one of $\overrightarrow e(u_l,L)>0$ or $\overleftarrow e(u_l,L)>0$ holds.
\end{claim}
This is obvious as otherwise we can get an extra copy of $P^*_3$.

\begin{claim}\label{iso2}
Let $u_{l_1}$ and $u_{l_2}$ be two isolated vertices from a copy of $P^*_3$ or $M_2$ in $\mathcal{H}$, which are adjacent to $L$. Define $i_j \equiv l_j \pmod{k}$ for $j \in [2]$, and assume $i_1 < i_2$. Then $i_2 - i_1 \le 2$. Moreover, if $e(u_{l_1}, L) + e(u_{l_2}, L) > |L|/k$, then $i_2 - i_1 = 1$.
\end{claim}

\begin{proof} 
Suppose that $i_2-i_1> 2$. Choose $v_{j_1}\in N(u_{l_1})\cap L$, $v_{j_2}\in N(u_{l_2})\cap L$. Then $M^{*}=v_{j_1}u_{l_1}\cup v_{j_2}u_{l_2}\cup(\{v_1,\ldots,v_k\}\setminus \{v_{j_1},v_{j_2},v_{l_1},v_{l_2}\})$ is a copy of $M_2$, where $v_i\in L_i$.
Hence, $\mathcal{H}\cup \{M^{*}\}$ is a $\{P_3^*,M_2\}$-tiling, contradicts the maximality of $\mathcal{H}$.

Now assume that $e(u_{l_1},L)+e(u_{l_2},L)>|L|/k$.
Suppose that $i_2-i_1= 2$.
By Claim \ref{iso}, we have at most one of $\overrightarrow e(u_l,L)>0$ or $\overleftarrow e(u_l,L)>0$ holds for any $l\in\{l_1,l_2\}$.
If $\overrightarrow{e}(u_{l_1},L)>0 $ and $\overleftarrow{e}(u_{l_2},L)>0$,  we have $\overrightarrow{e}(u_{l_1},L)+\overleftarrow{e}(u_{l_2},L)>|L|/k$.
Thus,
$ N(u_{l_1})\cap L\cap N(u_{l_2})\neq\emptyset$. 
Take $v_j\in N(u_{l_1})\cap L\cap N(u_{l_2})$. Then we obtain a copy $P=u_{l_1}v_ju_{l_2}\cup(\{v_1,\ldots, v_k\}\setminus \{v_{l_1},v_{j}, v_{l_2}\})$ of $P^*_3$, where $v_i\in L_i$. 
Hence, $\mathcal{H}\cup \{P\}$ is a $\{P_3^*,M_2\}$-tiling, contradicts the maximality of $\mathcal{H}$.
If at most one of $\overrightarrow{e}(u_{l_1},L)>0 $ and $\overleftarrow{e}(u_{l_2},L)>0$ holds, take $v_{j_1}\in N(u_{l_1}\cap L)$ and $v_{j_2}\in N(u_{l_2})\cap L$. Then $M^{**}=v_{j_1}u_{l_1}\cup v_{j_2}u_{l_2}\cup(\{v_1,\ldots,v_k\}\setminus \{v_{j_1},v_{j_2},v_{l_1},v_{l_2}\})$ is a copy of $M_2$, where $v_i\in L_i$.
Hence, $\mathcal{H}\cup \{M^{**}\}$ is a $\{P_3^*,M_2\}$-tiling, contradicts the maximality of $\mathcal{H}$.
\end{proof}

\begin{claim}\label{2edges}
    $e(J_N,L)\le 2|L|/k$ for every $N\in\{P^*,M,M'\}$. Moreover, there are at most $2$ isolated vertices $u_{l_1},u_{l_2}$ of $N$ with $e(u_{l_j},L)>0$ for any $j\in [2]$.
\end{claim}

\begin{proof}
Suppose that $e(J_N,L)> 2|L|/k$ for any $N\in\{P^*,M,M'\}$. By Claim \ref{iso}, we get that there are three isolated vertices $u_{l_1},u_{l_2},u_{l_3}$ with $e(u_{l_j},L)>0$ for every $j\in [3]$. 
Assume $i_j\equiv l_j \pmod{k}$ and $i_1<i_2<i_3$.
It is easy to check that $i_3-i_1\ge 2$.
By Claim \ref{iso2}, the proof is finished.
\end{proof}

From Claim \ref{2edges}, the following Claim can be directly obtained.

\begin{claim}\label{less2}
    If $e(I_N,L)\le 2|L|/k$ for any $N\in\{P^*,M,M'\}$, then   $e(J_N,L)\le 4|L|/k-e(I_N,L)$. 
\end{claim}

\begin{claim}\label{3edges}
    If $2|L|/k<e(I_{P^*},L) \leq 3|L|/k$, then $e(J_{P^*},L)\le 4|L|/k-e(I_{P^*},L)$.  
\end{claim}

\begin{proof}
 If only one isolated vertex $u_{l_j}\in V$  of $P^*$ with $e(u_{l_j},L)>0$, by Claim \ref{iso}, we have at most one of $\overrightarrow e(u_l,L)>0$ and $\overleftarrow e(u_l,L)>0$ holds, then $e(J_{P^*},L)+e(I_{P^*},L)\le 4|L|/k$. By Claim \ref{2edges}, we only need to consider that there are two isolated vertices $u_{l_1},u_{l_2}\in V$ of $P^*$ with $e(u_{l_j},L)>0$ for $j\in [2]$. Suppose that $e(J_{P^*},L)> 4|L|/k-e(I_{P^*},L)$. Let 
$i_j \equiv l_j \pmod{k}$ for $j\in [2]$.
Assume $i_1<i_2$. By Claim \ref{iso2}, we have $i_2-i_1\le 2$.
Choose $v_{l_i}\in N(u_{l_1})\cap L$, $v_{l_m}\in N(u_{l_2})\cap L$.
Then there are two possible cases.

{\bf Case 1.} At most one of $\overrightarrow e(u_i, L)>0$ and $\overleftarrow e(u_i, L)>0$ holds for any $i\in [2,4]$.
Since $i_2-i_1\le 2$, there exists at lesat one of $l_1$ and $l_2$, say $l_1$, satisfying $l_1\neq 1,5$.
If $\overrightarrow e(u_2,L)>0$ and $\overleftarrow e(u_4,L)>0$, then we have $N(u_2)\cap N(u_4)\cap L_3\neq\emptyset$ since $e(I_{P^*},L)> 2|L|/k$. Take $v_3\in N(u_2)\cap N(u_4)\cap L_3$, $v_{j_3}\in N(u_3)\cap L$. 
Then we can obtain a new $\{P^*_3,M_2\}$-tiling $\mathcal{H}'$ by replacing $P^*$ with $u_2v_3u_4\cup(\{v_{l_i},u_1,u_5,\cdots, u_k\}\setminus\{u_{l_i}\})$ and $u_{l_i}v_{l_1}\cup u_3v_{j_3}\cup(\{v_1,v_2,v_4,\cdots, v_k\}\setminus\{v_{l_1},v_{l_i},v_{j_3}\})$, where $v_s\in L_s$.  Clearly, $|V(\mathcal{H}')|=|V(\mathcal{H})|+k$ contradicts the maximality of $\mathcal{H}$.
If at most one of $\overrightarrow e(u_2,L)>0$ and $\overleftarrow e(u_4,L)>0$ holds, take $v_{j_2}\in N(u_2)\cap L$, $v_{j_3}\in N(u_3)\cap L$, $v_{j_4}\in N(u_4)\cap L$.
Then we can obtain a new $\{P^*_3,M_2\}$-tiling $\mathcal{H}''$ by replacing $P^*$ with $u_2v_{j_2}\cup u_4v_{j_4}\cup (\{v_{l_1},u_1,v_3,u_5,\cdots, u_k\}\setminus\{u_{l_i}\})$ and $v_{l_i}u_{l_1}\cup u_3v_{j_3}\cup(\{v_1,v_2,v_4,\cdots, v_k\}\setminus\{v_{l_1},v_{l_i},v_{j_3}\})$, where $v_s\in L_s$.  This contradicts the maximality of $\mathcal{H}$.

{\bf Case 2.} There exists a vertex $u_i\in \{u_2,u_3,u_4\}$ such that $e(u_i,L)=0$. By symmetry, we only need to consider the following two subcases.

{\bf Case 2.1.} $\overrightarrow e(u_2,L)>0$ and $\overleftarrow e(u_2,L)>0$. Choose $v_1\in N(u_2)\cap L_1$, $v_3\in N(u_2)\cap L_3$. By Claim \ref{P^*} (3), we have $\overrightarrow e(u_3,L)>0$ or $\overleftarrow e(u_4,L)>0$. 
Since $i_2-i_1\le 2$, there are at least one of $l_1,l_2$, say $l_1$, satisfying $l_1\neq 5$.
Then we can obtain a new $\{P^*_3,M_2\}$-tiling $\mathcal{H}''$ by replacing $P^*$ with $v_1u_2v_3\cup (\{v_{l_1},v_4,u_5,\cdots, u_k\}\setminus\{u_{l_i}\})$ and $v_{l_i}u_{l_1}\cup u_3u_4\cup(\{u_1,v_2,v_5,\cdots, v_k\}\setminus\{v_{l_1},v_{l_i}\})$, where $v_s\in L_s$. This contradicts the maximality of $\mathcal{H}$.

{\bf Case 2.2.} $\overrightarrow e(u_3,L)>0$ and $\overleftarrow e(u_3,L)>0$. By symmetry, we only need to consider $e(u_2,L)>0$. Choose $v_{j_2}\in N(u_2)\cap L$.
There are at least one of $l_1,l_2$, say $l_1$, satisfying $l_1\neq 1$.
Then we can obtain a new $\{P^*_3,M_2\}$-tiling $\mathcal{H}'$ by replacing $P^*$ with $v_2u_3v_4\cup (\{v_{l_1},u_1,u_4,u_5,\cdots, u_k\}\setminus\{u_{l_1}\})$ and $v_{l_i}u_{l_1}\cup u_2v_{j_2}\cup(\{v_1,v_3,u_4,v_5,\cdots, v_k\}\setminus\{v_{l_1},v_{l_i},v_{j_2}\})$, where $v_s\in L_s$.  This contradicts the maximality of $\mathcal{H}$.
\end{proof}

\begin{claim}\label{3edgesss}
    If $ 2|L|/k<e(I_M,L) \leq 3|L|/k$ and $ 2|L|/k<e(I_M',L) \leq 3|L|/k$, then  $e(J_M,L)\le 4|L|/k-e(I_M,L)$ and $e(J_{M'},L)\le 4|L|/k-e(I_{M'},L)$.  
\end{claim}

\begin{proof}

 If only one isolated vertex $u_{l_j}\in V$  of $M$ or $M'$ with $e(u_{l_j},L)>0$, by Claim \ref{iso}, we have at most one of $\overrightarrow e(u_l,L)>0$ and $\overleftarrow e(u_l,L)>0$ holds, then $e(J_M,L)+e(I_M,L)\le 4|L|/k$ and $e(J_M',L)+e(I_M',L)\le 4|L|/k$. By Claim \ref{2edges}, we only need to consider that there are two isolated vertices $u_{l_1},u_{l_2}\in V$ of $M$ or $M'$ with $e(u_{l_j},L)>0$ for $j\in [2]$. Suppose $e(J_M,L)> 4|L|/k-e(I_M,L)$ and $e(J_{M'},L)> 4|L|/k-e(I_{M'},L)$. Let 
$i_j \equiv l_j \pmod{k}$ for $j\in [2]$.
Assume that $i_1<i_2$. By Claim \ref{iso2}, we have $i_2-i_1\le 2$.
Choose $v_{l_i}\in N(u_{l_1})\cap L$, $v_{l_m}\in N(u_{l_2})\cap L$.

There is at least one of $l_1,l_2$, say $l_1$, satisfying $l_1\neq 1$.
Then we can obtain a new $\{P^*_3,M_2\}$-tiling $\mathcal{H}'$ either by replacing $M$ with $u_2u_3\cup\ u_{l_1}v_{l_i}\cup\ (\{v_{l_2},u_1,v_4,v_5,u_6,\cdots, u_k\}\setminus\{u_{l_1},u_{l_i}\})$ and $v_{l_m}u_{l_2}\cup u_4u_5\cup(\{u_{l_i},v_1,v_2,v_3,v_6,\cdots, v_k\}\setminus\{v_{l_m},v_{l_2}\})$, or by replacing $M'$ with $u_2u_3\cup\ u_{l_1}v_{l_i}\cup\ (\{v_{l_2},u_1,u_4,v_5,v_6,u_7,\cdots, u_k\}\setminus\{u_{l_1},u_{l_i}\})$ and $v_{l_m}u_{l_2}\cup u_5u_6\cup(\{u_{l_i},v_1,v_2,v_3,v_4,v_7,\cdots, v_k\}\setminus\{v_{l_m},v_{l_2}\}\})$, where $v_s\in L_s$. This contradicts the maximality of $\mathcal{H}$.
\end{proof}

To finish up, we proceed by case analysis on the number of edges between non-isolated vertices of any copy of $P^*_3$ or $M_2$ and $L$, as detailed below.
For any copy $P^*$ of $P^*_3$,
Claims~\ref{more5}, \ref{4edges}, \ref{3edges} and \ref{less2} altogether ensure that
$e(I_{P^*},L)+e(J_{P^*},L)\le 4|L|/k$ holds for every possible value of $e(I_{P^*},L)$.
Similarly,
let $M$, $M'$ be any two copies of $M_2$. For any $N\in\{M,M'\}$, Claims~\ref{more5M}, \ref{4edgesM}, \ref{3edgesss} and \ref{less2} together imply that $e(I_{N},L)+e(J_{N},L)\le 4|L|/k$ holds in various cases on the size of $e(I_{N},L)$. This completes the proof of Fact \ref{0485}.

\end{document}